\documentclass[11pt]{amsart}

\usepackage{color}

\usepackage{amsmath,amssymb,amsthm,amsfonts,amsbsy,amstext}

\usepackage{amsmath, amssymb} 
\usepackage{amsthm, amsfonts, mathrsfs}

\usepackage{amsfonts,graphicx}

\newcommand{\DD}{\textnormal{D}}

\theoremstyle{plain}
\newtheorem{Theo}{Theorem}[section]
\newtheorem{lem}[Theo]{Lemma}

\newtheorem{prop}[Theo]{Proposition}

\theoremstyle{plain} \theoremstyle{definition}
\newtheorem{defi}[Theo]{Definition}

\theoremstyle{remark}
\newtheorem{Rema}[Theo]{Remark}
\newtheorem*{rema*}{Remark}

\newcommand{\ZZ}{\mathbb{Z}}

\newcommand{\NN}{\mathbb{N}}

\newcommand{\RR}{\mathbb{R}}

\numberwithin{equation}{section}

\date{}

\begin{document}

\title[Global well-posedness for Boussinesq system ]
{On the global existence for the axisymmetric Euler-Boussinesq system in critical Besov spaces}
\author[Samira Sulaiman]{Samira Sulaiman}
\address{IRMAR, Universit\'e de Rennes 1\\ Campus de
Beaulieu\\ 35~042 Rennes cedex\\ France}
\email{samira.sulaiman@univ-rennes1.fr}

\begin{abstract}
This paper is devoted to the global existence and uniqueness results for the three-dimensional Boussinesq system with axisymmetric initial data $v^{0}\in B^{\frac{5}{2}}_{2,1}(\RR^3)$ and $\rho^{0}\in B^{\frac{1}{2}}_{2,1}(\RR^3)\cap L^{p}(\RR^3)$ with $p>6$. This system couples the incompressible Euler equations with a transport-diffusion equation governing the density. In this case the Beale-Kato-Majda criterion (see \cite{bkm84}) is not known to be valid and to circumvent this difficulty we use in a crucial way some geometric properties of the vorticity. 
\end{abstract}

\maketitle
\tableofcontents

\section{Introduction and main result}\label{A} The present paper is devoted to the mathematical study of the so-called Boussinesq system. This system couples the incompressible Euler equations and a transport-diffusion equation governing the density. It reads as follows 
\begin{equation}\label{a} 
\left\{ \begin{array}{ll} 
\partial_{t}v+v\cdot\nabla v+\nabla p=\rho\,e_{z},\qquad (t,x)\in\RR_+\times\RR^3\\ 
\partial_{t}\rho+v\cdot\nabla\rho-\kappa\Delta\rho=0\\
\textnormal{div}\,v=0\\
v_{| t=0}=v^{0}, \quad \rho_{| t=0}=\rho^{0}.  
\end{array} \right.
\end{equation} 
Here, the velocity $v=(v^1,v^2,v^3)$ is three-component vector field with zero divergence. The scalar function $\rho$ denotes the density which is diffused by the flow and acts on the first equation of \eqref{a} in the vertical direction $e_{z}=(0,0,1).$ The pressure $p$ is a scalar function which given by the equation
$$\Delta P=-\textnormal{div}(v\cdot\nabla v)+\partial_{z}\rho.$$ 
The coefficient $\kappa\ge0$ is a heat conductivity and we will take in the sequel $\kappa=1.$ The term $v\cdot\nabla$ is defined by $$v\cdot\nabla=\displaystyle\sum^{3}_{i=1}v^{i}\partial_{i}.$$
The system \eqref{a} is a special case of a class of generalized Boussinesq system introduced in \cite{YB08}. Note that if the initial density $\rho^{0}$ is identically zero (or constant), then the system \eqref{a} reduces to the classical incompressible Euler equations :
\begin{equation}\label{b} 
\left\{ \begin{array}{ll} 
\partial_{t}v+v\cdot\nabla v+\nabla p=0\\ 
\textnormal{div}\,v=0\\
v_{| t=0}=v^{0}.  
\end{array} \right.
\end{equation} 
In three dimension, the breakdown of smooth solutions to the Boussinesq system \eqref{a} and Euler equations \eqref{b} remains an open problem. For the case of axisymmetric solutions without swirl (see definition below), global existence and uniqueness for \eqref{a} and \eqref{b} has been established under various assumptions on the initial data.
\begin{defi}
We say that a vector field $v$ is axisymmetric (without swirl) if and only if it has the form :
\begin{equation*}
v(t,x)=v^{r}(t,r,z) e_{r}+v^{z}(t,r,z) e_{z},
\end{equation*}
where $z=x_3\;,\; x=(x_1, x_2, z)\;,\;r=(x_{1}^{2}+x_{2}^{2})^{\frac{1}{2}}\;\;\textnormal{and}\;\;(e_r, e_\theta, e_z)$ is the cylindrical basis of $\RR^3$ given by :
\begin{equation*}
e_r=(\frac{x_{1}}{r}, \frac{x_{2}}{r},0)\quad e_\theta=(-\frac{x_{2}}{r}, \frac{x_{1}}{r},0)\quad\textnormal{and}\;\;\;e_z=(0,0,1)
\end{equation*}
and the components $v^r$ and $v^z$ do not depend on the angular variable $\theta.$
\end{defi}
Recall now that in cylindrical coordinates, we have:
\begin{equation*}\label{qz5}
v\cdot\nabla=v^{r}\partial_{r}+\frac{1}{r}v^{\theta}\partial_{\theta}+v^{z}\partial_{z}
\end{equation*}
and
$$\textnormal{div}\,v= \partial_{r}v^{r}+\frac{v^r}{r}+\partial_{z}v^{z}.$$
Thus for an axisymmetric vector field (without swirl) i.e $v^\theta=0,$ we get,
\begin{equation}\label{f}
v\cdot\nabla=v^{r}\partial_{r}+v^{z}\partial_{z}.
\end{equation}
In space dimension three, the vorticity $\omega$ of $v$ is defined as the vector $\omega=curl v=\nabla\times v$ and has the form :
\begin{equation}\label{c}
\omega:=\omega^{\theta}e_{\theta}\quad\textnormal{with}\quad \omega^{\theta}=\partial_{z} v^{r}-\partial_{r} v^{z}.
\end{equation}
Moreover, from \eqref{a} the vorticity $\omega$ of $v$ satisfies the equation,
$$\partial_{t}\omega+(v\cdot\nabla)\omega-(\omega\cdot\nabla)v=curl(\rho e_{z}).$$
Similarly, for the system \eqref{b}, the vorticity satisfies \eqref{c} and
\begin{equation}\label{d}
\partial_{t}\omega+(v\cdot\nabla)\omega-(\omega\cdot\nabla)v=0.
\end{equation}
It is well-known, according to the Beale-Kato-Majda criterion \cite{bkm84} that the control of the vorticity in $L^\infty$ is sufficient to get global well-posedness results for smooth initial data $v^{0}\in H^{s}, s>\frac{5}{2}$. The main difficulty to bound the vorticity is the lack of information about the influence of the vortex-stretching term $\omega\cdot\nabla v$ on the motion of the fluid.\\
In the case of axisymmetric flows without swirl, we have a good behavior of the stretching term. It takes the form  
$$\omega\cdot\nabla v=\frac{v^r}{r}\omega.$$
and thus the vorticity equation becomes 
\begin{equation}\label{g}
\partial_{t}\omega+(v\cdot\nabla)\omega=\frac {v^r}{r}\omega.
\end{equation}
Letting $\beta:=\frac{\omega}{r}$ in \eqref{g}, the quantity $\beta$ then solves the equation
$$\partial_{t}\beta+(v\cdot\nabla)\beta=0.$$
From the incompressibility of the velocity $v,$ we get easily for $p \in[1,\infty]$
$$\Vert\beta(t)\Vert_{L^p}\le \Vert\beta^0\Vert_{L^p}.$$ 
In \cite{ui68}, Ukhovskii and Iudovich took advantage of these conservation laws to prove global existence for axisymmetric initial data with finite energy for the Euler system \eqref{b} and satisfying in addition $\omega^{0}\in L^{2}\cap L^\infty$ and $\frac{\omega^0}{r} \in L^{2}\cap L^\infty.$ This result was improved by Shirota and Yanagisawa \cite{sy94} in $H^{s}$ with $s>\frac{5}{2}.$ Their proof is based on the boundness of the quantity $\Vert\frac{v^r}{r}\Vert_{L^\infty}$ by using some Biot-Savart law. \\
In \cite{sr94} similar results are given in different function spaces. Danchin \cite{d07} has weakened the Ukhovskii and Iudovich condition for initial data $\omega^{0}\in L^{3,1}\cap L^\infty$ and $\frac{\omega^0}{r} \in L^{3,1},$ where $L^{3,1}$ denotes a Lorentz space (see definition \ref{def2} in section \ref{B}). \\
In \cite{ahs08} Abidi, Hmidi and Keraani proved that if $v^{0}$ is an axisymmetric vector field in the critical Besov space $B^{\frac{3}{p}+1}_{p,1}(\RR^3)$ with $p\in[1,\infty]$ and if $\frac{\omega^{0}}{r}\in L^{3,1}(\RR^3)$, then there exists a unique global solution $v$ to \eqref{b} in $\mathcal{C}(\RR_+;B_{p,1}^{\frac{3}{p}+1}(\RR^3)).$ In this context of critical regularities, we do not know whether the B-K-M criterion is applicable. To avoid this difficulty, the authors obtained first the $L^\infty$ bound of the vorticity for every time by using a kind of Biot-Savart law. Second, they established a new estimate for the vorticity in the Besov space $B_{\infty,1}^{0}$ which is based on the analysis of the geometric structure of the vorticity and paraproduct calculus. We point out that in dimension 2, the problem of global well-posedness in critical Besov space $B_{p,1}^{\frac{2}{p}+1}$ was solved by Vishik \cite{vis98}. The crucial ingredient of Vishik's proof is a logarithmic estimate in Besov space $B_{\infty,1}^{0}$ for the composition $f\circ\psi,$ where $f\in B_{\infty,1}^{0}$ and $\psi$ is the flow of the velocity $v$ which preserves the Lebesgue measure.
As we have already seen that the system \eqref{a} is a perturbation  of Euler system \eqref{b} then it is legitimate to extend the global well-posedness results for the system \eqref{a}. Let us focus on the difficulties for \eqref{a} when we try to get some a priori estimates for the vorticity.
Recall that the vorticity of $v$ for \eqref{a} satisfies the equation,
$$\partial_{t}\omega+v\cdot\nabla\omega=\frac{v^{r}}{r}\omega+curl(\rho e_{z}),$$
with simple calculations, we get 
\begin{equation}\label{aa}
curl(\rho e_z)=\nabla\times(\rho e_{z})=
\begin{pmatrix}
\partial_2 \rho\\
-\partial_1 \rho\\
0    
\end{pmatrix} =-(\partial_r \rho)e_\theta.
\end{equation}
Therefore the vorticity obeys the equation
$$\partial_{t}\omega+v\cdot\nabla\omega=\frac {v^r}{r}\omega-(\partial_{r}\rho)e_{\theta}$$
and the equation of the scalar component of the vorticity is given by
$$\partial_{t}\omega_\theta+(v\cdot\nabla)\omega_\theta=\frac{v^r}{r}\omega_{\theta}-\partial_{r}\rho.$$
It follows that the evolution of the quantity $\frac{\omega_{\theta}}{r}$ is governed by the equation
\begin{equation}\label{h}
(\partial_{t}+v\cdot\nabla)\frac{\omega_{\theta}}{r}=\frac{-\partial_{r}\rho}{r}.
\end{equation}
The main difficulty is to find some a priori estimates on the density $\rho$ to control the right-hand side of \eqref{h}. The idea is that the singularity $\frac{1}{r}$ on the axis $r=0$ is a derivative and that the term $\frac{\partial_{r}\rho}{r}$ can be thought of as a Laplacian of the density $\rho.$ Thus the authors in \cite{tf010} try to use some smoothing effects to control this term $\frac{\partial_{r}\rho}{r}.$ We refer to Proposition 4.4 in \cite{tf010}, since they introduced a function $\Gamma:=\frac{\omega_\theta}{r}+\frac{\partial_r}{r}\Delta^{-1}\rho$ and studied the coupling system $(\Gamma, \rho)$ to find an estimate for the $L^\infty$ norm of $\frac{v^r}{r}$. They used this to prove existence and uniqueness of a solution to the Boussinesq system \eqref{a} with axisymmetric initial data $v^{0}\in H^{s}(\RR^3)$, $\frac{5}{2} < s <3$ and $\rho^{0}\in H^{s-2}(\RR^3)\cap L^{m}(\RR^3)$, $6<m$ with $\vert x_{h}\vert^{2}\rho^{0}\in L^{2}$, $x_{h}=(x_1, x_2).$ 
The goal of this paper is to improve these results by weakening the initial reqularities in order to allow critical Besov spaces. Our main result isthe following :
\begin{Theo}\label{Theo1}
Let $v^0 \in B^{\frac{5}{2}}_{2,1}$ be an axisymmetric vector field with zero divergence without swirl and $\rho^0$ be an axisymmetric function such that $\rho^{0}\in B^{\frac{1}{2}}_{2,1}\cap L^p$ with $p>6$ and such that $\vert x_{h}\vert^{2}\rho^{0}\in L^{2}.$ Then there exists a unique global solution $(v,\rho)$ for the system \eqref{a} such that
\begin{equation*}
v \in\mathcal{C}(\RR_+;B^{\frac{5}{2}}_{2,1}),\;\;\rho\in \mathcal{C}(\RR_+;B^{\frac{1}{2}}_{2,1}\cap L^p)\cap L_{loc}^1\big(\RR_+; Lip \big)\;\;\textnormal{and}\;\;\vert x_{h}\vert^{2}\rho\in\mathcal{C}(\RR_+; L^2).
\end{equation*}
\end{Theo}
The definition of Besov spaces is recalled in the next section.
\begin{Rema}\label{bb}
Since we have for every $1\le p\le 2$ the Besov embedding $B^{\frac{3}{p}+1}_{p,1}\hookrightarrow B^{\frac{5}{2}}_{2,1},$ then the above result remain true if we replays $B^{\frac{5}{2}}_{2,1}$ by $B^{\frac{3}{p}+1}_{p,1}.$ For $p>2$ the method does not work because we need some energy estimates for the velocity.
\end{Rema}
\begin{Rema}\label{cc}
Our proof gives an integrability in time of the density $\rho.$ More precisely we have $\rho\in L^{1}_{loc}(\RR_+; B^{\frac{5}{2}}_{2,1}(\RR^3))$ (see Proposition \ref{prop a7} ).
\end{Rema}
The proof relies essentially on two crucial estimates. The first one is a global a priori estimates of the vorticity in $L^\infty$ space which is based on Biot-Savart law with some estimates established in \cite{tf010} (see Proposition \ref{prop a3}). Nevertheless, this information is not sufficient to propagate the initial regularities because the Beale-Kato-Majda criterion \cite{bkm84} is not know to be valid. The significant quantity that one should estimate is $\Vert\omega\Vert_{B_{\infty,1}^{0}}.$ To reach our goal we use an approach developed in\cite{ahs08} (see Proposition \ref{prop a4}) witch is the hard part of the proof in the paper where the axisymmetric geometry plays a crucial role. This allows to bound for every time the Lipschitz norm of the velocity and then to propagate the regularities.\\ 
The paper is structured as follows : In Section \ref{B}, we fix some notations, we recall some basic tools from Littlewood-Paley theory and we will introduce some function spaces. We also state a few useful estimates for a transport-equation that we will use later. In Section \ref{C}, we study some geometric properties of any solution to a vorticity equation model. The proof of Theorem 1.2 is done in a several steps in section \ref{D}. 

\section{Preliminaries}\label{B}
In this section, we introduce some notations and definitions of Besov spaces. We give also some results about Lorentz space and recall some well-known results about the Littlewood-Paley decomposition and transport-diffusion equation used later. Let us begin with notations.
\subsection{Notation} We will use the following notations :\\
$\bullet$ For any positive $G$ and $H$, the notation  $G\lesssim H$ means that there exists a positive constant $C$ independent of $G$ and $H$ and such that $G\leqslant CH$.\\
$\bullet$ For any tempered distribution $g,$ both $\widehat{g}$ and $\mathcal{F}(g)$ denote the Fourier transform of $g$ with
$$\widehat{g}(\xi)=\int_{\RR^{3}} g(x) e^{-ix\xi} dx.$$
$\bullet$ For any pair of operator C and $D$ on some Banach space $\mathcal{A}$, the commutator $[C,D]$ is defined by $CD-DC.$\\
$\bullet$ The space $L^p$, $1\le p\le \infty$ stands for the usual Lebesgue space and then for any Banach space $Z$ with norm $\Vert\cdot\Vert_Z$ and function $h(t,x)$ such that for every $t,$ $h(t,x)\in Z$, we shall use the notation 
$$\Vert h \Vert_{L_T^p Z}=(\displaystyle\int_0^T\Vert h(\tau) \Vert_Z^p d\tau)^\frac{1}{p},\;\forall\;\; T>0.$$\\
$\bullet$ We denote by $\dot{W}^{1,p}$ with $1\le p\le \infty$ the space of distribution $f$ such that $\nabla f \in L^p.$\\
$\bullet$ We will use also,
$$V(t):=\displaystyle\int_0^t\Vert v(\tau)\Vert_{B^{1}_{\infty,1}}d\tau.$$
$\bullet$ We will introduce the following notation : we denote by
$$\Phi_{l}(t)=C_{0}\underbrace{\exp(...\exp}_{l-times}(C_{0} t^{\frac{19}{6}})...),$$ where $C_{0}$ depends on the initial data and its value may vary from line to line up to some absolute constants. We will make an intensive use of the following trivial facts
\begin{equation*}
\int^{t}_{0}\Phi_{l}(\tau) d\tau \le \Phi_{l}(t)\qquad\textnormal{and}\qquad \exp(\int^{t}_{0}\Phi_{l}(\tau) d\tau)\le \Phi_{l+1}(t).
\end{equation*}
\subsection{Littlewood-Paley decomposition and Besov space}
To introduce Besov spaces which are generalization of Sobolev spaces we need to recall the dyadic decomposition of the whole space (see Chemin \cite{che98}).
\begin{prop}\label{prop1} There exists two nonnegative radial functions $\chi\in\mathscr{D}(\mathbb{R}^{3})$ and $\varphi\in\mathscr{D}(\mathbb{R}^{3}\backslash\{0\})$ such that, 
$$\chi(\xi)+\displaystyle \sum_{j\ge 0}\varphi(2^{-j}\xi)=1, \quad\forall \xi\in\mathbb{R}^{3},$$
$$\displaystyle \sum_{j\in\mathbb{Z}}\varphi(2^{-j}\xi)=1, \quad\forall \xi\in\mathbb{R}^{3}\backslash\{0\},$$
$$\vert p-j\vert\ge 2\Rightarrow\mbox{supp }{\varphi}(2^{-p}\cdot)\cap\mbox{supp }{\varphi}(2^{-j}\cdot)=\varnothing,$$
$$j\ge 1\Rightarrow \mbox{supp }{\chi}\cap\mbox{supp }{\varphi}(2^{-j}\cdot)=\varnothing.$$
\end{prop}
Set $\varphi_{j}(\xi)=\varphi(2^{-j}\xi)$ and let $h=\mathcal{F}^{-1}\varphi$ and $\bar{h}=\mathcal{F}^{-1}\chi.$ Define the frequency localization operators $\Delta_{j}$ and $S_{j}$ For every $f\in\mathcal{S}^\prime(\RR^3)$ by
\begin{equation*}
\Delta_{j}f=\varphi(2^{-j}\DD)f= 2^{3\,j}h (2^{j}\cdot) \ast f\;\;,\,\textnormal{for}\;\; j\geqslant 0,
\end{equation*}
\begin{equation*}
S_{j}f=\chi(2^{-j}\DD)f =\displaystyle \sum_{-1\le p\le j-1}\Delta_{p}f= 2^{3\,j} \bar{h}(2^{j}\cdot) \ast f ,
\end{equation*}
\begin{equation*}
\Delta_{-1}f=S_{0}f ,\qquad \Delta_{j}f=0 \qquad \textnormal{for}\;\;\; j\le-2.
\end{equation*}\\
One can easily prove that for every tempered distrubution $v$, we have
\begin{equation}\label{i}
v(x)=\sum_{j\ge-1}\Delta_{j}v(x).
\end{equation}
The homogeneous operators are defined as follows
$$\forall j \in \ZZ;\,\dot{\Delta}_{j}v=\varphi(2^{-j}\DD)v\;\;\textnormal{and}\;\;\dot{S}_j v=\sum_{k \le j-1}\dot{\Delta}_{k}v.$$
From the homogeneous decomposition, we have
$$v=\displaystyle \sum_{j\in\ZZ}\dot{\Delta}_{j}v,\qquad \forall v\in \mathcal{S}^{\prime}(\RR^3)/\mathcal{P}(\RR^3),$$
where $\mathcal{P}(\RR^3)$ is the space of all polynomials see \cite{jp76}.\\
From the paradifferential calculus introduce by J.-M. Bony \cite{bo81} the product $vw$ can be formally divided into three parts as follows :
\begin{equation}\label{j}
vw=T_v w+T_wv+R(v,w),
\end{equation}
where
\begin{equation*}
T_vw=\sum_{j}S_{j-1}v\Delta_{j}w\;\;\textnormal{and}\;\;R(v,w)=\sum_{j}\Delta_{j} v \widetilde{\Delta}_{j}w,
\end{equation*}
$$\textnormal{with}\:\:\:\widetilde{\Delta}_{j}=\sum_{i=-1}^{1}\Delta_{j+i}.$$
$T_{v}w$ called paraproduct of $w$ by $v$ and $R(v,w)$ the reminder term.\\ 
Recall now the following definition of general Besov spaces.
\begin{defi}\label{def1}
Let $s\in\RR$ and $1 \le p,r \le +\infty.$ The inhomogeneous Besov space $B_{p,r}^s$ is defined by
$$B^{s}_{p,r}=\left\lbrace f \in \mathcal{S}^{\prime}(\RR^{3}):\Vert f \Vert_{B^{s}_{p,r}}<\infty \right\rbrace,$$
where 
\begin{equation*}\|f\|_{B_{p,r}^s}:=
\left\{\begin{array}{ll} 
\bigg(\displaystyle  \sum_{j \geq 0} 2^{jsr} \|\Delta_j f\|^{r}_{L^{p}}\bigg)^\frac{1}{r}+\Vert S_{0}f \Vert_{L^{p}}  \;\;\textnormal{for}\;\; r<\infty,\\
\displaystyle \sup_{j \geq 0} 2^{js} \|\Delta_j f\|_{L^{p}}+\Vert S_{0}f \Vert_{L^{p}} \;\;\textnormal{for}\;\; r=\infty. 
\end{array} \right.
\end{equation*}
The homogeneous norm
\begin{equation*}\Vert f \Vert_{\dot{B}_{p,r}^{s}}:= 
\left\{\begin{array}{ll} 
\bigg(\displaystyle \sum_{j \in \ZZ} 2^{jsr} \Vert\dot{\Delta}_{j} f \Vert^{r}_{L^{p}}\bigg)^\frac{1}{r}\;\;\textnormal{for}\;\; r<\infty\\
\displaystyle \sup_{j \in \ZZ} 2^{js} \Vert\dot{\Delta}_{j} f \Vert_{L^{p}}\;\;\textnormal{for}\;\; r=\infty.  
\end{array} \right.
\end{equation*} 
\end{defi}
We have the following Bernstein inequality see Chemin \cite{che98}. 
\begin{lem}\label{Bernstein} There exists a constant $C>0$ such that for $1\le p_{1}\le p_{2}\le\infty$, $k\in\NN$, $j\in\ZZ$ and for  every function $v$ we have
\begin{eqnarray*}
\sup_{\vert\alpha\vert=k}\Vert\partial^{\alpha}S_{j}v\Vert_{L^{p_{2}}}&\le& C^{k}2^{j\big(k+3\big(\frac{1}{p_{1}}-\frac{1}{p_{2}}\big)\big)}\Vert S_{j}v \Vert_{L^{p_{1}}},\\
C^{-k}2^{jk}\Vert\dot{\Delta}_{j}v\Vert_{L^{p_{1}}}&\le&\sup_{\vert\alpha\vert=k}\Vert\partial^{\alpha}\dot{\Delta}_{j}v \Vert_{L^{p_{1}}}\le C^{k}2^{jk}\Vert \dot{\Delta}_{j}v\Vert_{L^{p_{1}}}.
\end{eqnarray*}
\end{lem}
Using this Lemma and the definition above of Besov space we have $\forall v \in \mathcal{S}^{\prime}$, $s \in \RR$ and $p_{1} \le p_{2},$
\begin{eqnarray*}
\Vert v \Vert_{B^{s+3(\frac{1}{p_{2}}-\frac{1}{p_{1}})}_{p_{2},1}}&=& \sum_{j\ge-1} 2^{j(s+3(\frac{1}{p_{2}}-\frac{1}{p_{1}}))}\Vert\Delta_{j} v \Vert_{L^{p_2}}\\
&\lesssim& \sum_{j\ge-1} 2^{js}\Vert\Delta_{j} v \Vert_{L^{p_1}}.
\end{eqnarray*}
This gives the embedding 
$$B^{s}_{p_{1},1}\hookrightarrow B^{s+3(\frac{1}{p_{2}}-\frac{1}{p_{1}})}_{p_{2},1}\quad,\; p_1\le p_2.$$
We can easily then generalise  this embedding for all $r_{1}\le r_{2}$ that is we have,
\begin{equation}\label{k}
B_{p_1,r_1}^{s}\hookrightarrow B_{p_2,r_2}^{s+3(\frac{1}{p_2}-\frac{1}{p_1})}\,,\;\,p_1\le p_2\;\;\textnormal{and}\;\; r_1 \le r_2.
\end{equation}
\subsection{Lorentz spaces}
Before introduce the definition of the Lorentz spaces, we recall the nonincreasing rearrangement. Let $h$ a measurable function we define its nonincreasing rearrangement $h^{\ast}:\RR_{+}\longrightarrow \RR_{+}$ by the formula :
\begin{equation*}
h^{\ast}(t):= \inf\Big\{s^{\prime}\ge 0; l(\{y,\vert h(y)\vert > s^{\prime}\})\le t\Big\},
\end{equation*}
where $l$ denote the usual Lebesgue measure.
\begin{defi}\label{def2}(Lorentz space)
Let $h$ a measurable function and $1 \le p\le \infty.$ Then $h$ belong to the space of Lorentz if 
\begin{equation*}\Vert h \Vert_{L^{p,r}}:= 
\left\{\begin{array}{ll} 
\big(\displaystyle\int^{\infty}_{0}(t^{\frac{1}{p}}h^{\ast}(t))^{r}\frac{dt}{t}\big)^{\frac{1}{r}}<\infty\;\;\textnormal{if}\;\;1\le r<\infty\\
\displaystyle \sup_{t>0}t^{\frac{1}{p}}h^{\ast}(t)\;\;\textnormal{if}\;\;r=\infty.  
\end{array} \right.
\end{equation*} 
\end{defi}   
We can also define Lorentz spaces by interpolation from Lebesgue spaces :
$$(L^{p_1},L^{p_2})_{(\mu,r)}=L^{p,r},$$
where $1 \le p_1 \le p_2 \le \infty$, $\mu$ satisfies $\frac{1}{p}=\frac{1-\mu}{p_1}+\frac{\mu}{p_2}$ and $1\le r\le\infty.$\\
From this definition, we get
\begin{equation*}
L^{p,r}\hookrightarrow L^{p,r_1}\;\;,\;\forall 1 \le p \le \infty,\;1 \le r \le r_1 \le \infty\qquad\textnormal{and}\;\;\;L^{p,p}=L^p.
\end{equation*}

\subsection{About a transport equation}
We will give now some useful estimates for any smooth solution of linear transport-diffusion model given by
\begin{equation}\label{dd} 
\left\{ \begin{array}{ll}  
\partial_{t}h+v\cdot\nabla h-\kappa\Delta h=g\\
h_{| t=0}=h^{0} 
\end{array} \right.
\end{equation} 
 We will give three kind of estimates : the first is the $L^p$ estimate, the second describes the propagation of Besov regularity and the third is the smoothing effects.
\begin{lem}\label{m} Let $v$ be a smooth divergence-free vector field of $\RR^3$ and $h$ be a smooth solution of \eqref{dd}. Then we have $\forall p \in [1,\infty]$ and for every $\kappa \ge 0,$  
$$\Vert h(t)\Vert_{L^p}\le \Vert h^{0}\Vert_{L^p}+\int_{0}^{t}\Vert g(\tau)\Vert_{L^p}d\tau.$$ 
\end{lem} 
\begin{proof}
For finite $p$, multiplying the first equation of \eqref{dd} by $\vert h \vert^{p-2} h$, integrating by parts and using the condition of divergence free for the vector field, we find the estimation of Lemma. While for $p=\infty$ it is just the maximum principle see \cite{h05} for more details. 
\end{proof}
\begin{prop}\label{prop2} Let $-1 < s <1$, $1 \le p,r \le \infty$ and $v$ be a smooth divergence-free vector field. Let $h$ be a smooth solution of \eqref{dd} such that $h^0 \in B_{p,r}^s(\RR^3)$ and $g \in L^1_{loc}(\RR_+,B_{p,r}^s).$ Then $\forall t \in \RR_+,$ we have with $\kappa=0,$
\begin{equation*}
\Vert h(t) \Vert_{B_{p,r}^s}\lesssim e^{C V_{1}(t)}(\Vert h^0 \Vert_{B_{p,r}^s}+\int_0^t e^{-C V_{1}(\tau)}\Vert g(\tau) \Vert_{B_{p,r}^s} d\tau),
\end{equation*}
where $C$ is a constant depending on $s$ and $V_{1}(t):=\Vert\nabla v(t) \Vert_{L_t^1L^\infty}.$\\ The above estimate is true in the two following case
$$s=-1, r=\infty, 1 \le p \le \infty\;\textnormal{and}\;\; s=r=1, 1 \le p \le \infty,$$
in this case we change $V_{1}(t)$ by $V(t):=\Vert v(t) \Vert_{L_t^1 B_{\infty,1}^1}.$ 
\end{prop}
Abidi, Hmidi and Keraani \cite{ahs08} are proved this result for the case $s=\pm1$. The remainder cases are done in \cite{che98, vis98}.\\  
We give now the following smoothing effects for a transport equation \eqref{dd} with respect to a vector field which is not necessary Lipschitzian (see \cite{hk1} for a proof).
\begin{prop}\label{prop3} Let $v$ be a smooth divergence-free vector field of $\RR^3$ with vorticity $\omega:= \nabla\times v$. Let $h$ be a smooth solution of \eqref{dd}. Then for $h^0 \in L^p$, $1 \le p \le \infty$, $\kappa>0$ and $t\in \RR_+,$ we have for all $j \in\NN$ with $(g=0)$
$$2^{2j} \displaystyle \int_0^t \Vert \Delta_{j} h(\tau)\Vert_{L^p} d\tau \lesssim \Vert h^0 \Vert_{L^p}\Big(1+(j+1)\displaystyle \int_0^t \Vert \omega(\tau)\Vert_{L^\infty} d\tau+\displaystyle \int_0^t \Vert\nabla\Delta_{-1} v(\tau)\Vert_{L^\infty} d\tau\Big).$$ 
\end{prop}
On the other-hand, Hmidi \cite{h05} proved the following smoothing effects for the equation \eqref{dd} with $(g=0)$ in the case of Lipschitzian velocity,
\begin{equation}\label{0x}
2^{2j} \displaystyle \int_0^t \Vert \Delta_{j} h(\tau)\Vert_{L^p} d\tau \lesssim \Vert h^0 \Vert_{L^p}\Big(1+\int_0^t \Vert \nabla v(\tau)\Vert_{L^\infty} d\tau\Big). 
\end{equation}

\section{Study of Geometric Properties of the vorticity}\label{C}
The aim of this section is the study of some geometric properties of axisymmetric flows. We start with the following one which is proved in \cite{ahs08}. 
\begin{prop}\label{prop4} Let $v=(v^1,v^2,v^3)$ be a smooth axisymmetric vector field. Then we have\\
\begin{itemize}
\item For every $j \ge -1$, $\Delta_j v$ is axisymmetric and 
$$(\Delta_jv^1)(0,x_2,z)=(\Delta_jv^2)(x_1,0,z)=0.$$
\item The vector $\omega=\nabla\times v =(\omega^1,\omega^2,\omega^3)$ satisfies $\omega\times e_\theta=(0,0,0)$ and we have for every $(x_1,x_2,z) \in \RR^3,$ 
$$\omega^3=0,\;\;\;x_1 \omega^1(x_1,x_2,z)+ x_2 \omega^2(x_1,x_2,z)=0$$ and
$$\omega^1(x_1,0,z)= \omega^2(0,x_2,z)=0.$$
\end{itemize}
\end{prop}
We aim now to study a vorticity like equation :
\begin{equation}\label{n}
\left\{
\begin{array}{ll} 
\partial_{t}\Omega+v\cdot\nabla \Omega=\Omega\cdot\nabla v+curl(\rho e_{z})\\ 
\textnormal{div}\, v=0\\
\Omega_{| t=0}=\Omega^{0},
\end{array} \right.
\end{equation} 
and we will assume that $v$ and $\rho$ are axisymmetric, the unknown function $\Omega=(\Omega^1,\Omega^2,\Omega^3)$ is a vector field and at this stage we do not assume that $\Omega$ is the vorticity of $v.$ But $\Omega$ has some geometric properties with the vorticity $\omega.$ More precisely we prove the following Proposition which describes the preservation of some initial geometric conditions of $\Omega.$
\begin{prop}\label{prop5} Let $v$  be a smooth axisymmetric vector field with divergence free and $\Omega$ be the unique global solution of \eqref{n} with smooth initial data $\Omega^0.$ Then we have the following properties.\\
$1)$ If $\textnormal{div}\, \Omega^0=0$, then $\textnormal{div}\, \Omega(t)=0,\;\,\forall\,t\in\RR_+.$\\
$2)$ If $\Omega^0\times e_\theta=(0,0,0)$, then we have $\forall\, t\in\RR_+$
$$\Omega(t)\times e_\theta=(0,0,0).$$ 
Consequently $\Omega^1(t,x_1,0,z)=\Omega^2(t,0,x_2,z)=0$ and 
$$\Omega\cdot\nabla v=\frac {v^r}{r}\Omega.$$
\end{prop}
\begin{proof} $1)$ We apply the divergence operator to the equation \eqref{n}, we get
$$\partial_{t}\textnormal{div}\,\Omega+\textnormal{div}(v\cdot\nabla\Omega)=\textnormal{div}(\Omega\cdot\nabla v)+\textnormal{div}(curl(\rho e_{z})).$$
Now we have $\textnormal{div}(curl(\rho e_{z}))=0$ and using the fact that $\textnormal{div}\,v=0,$ we get
\begin{eqnarray*}
\textnormal{div}(v\cdot\nabla\Omega)-\textnormal{div}(\Omega\cdot\nabla v)&=& \sum_{i,j=1}^{3}\partial_{j}(v^{i}\partial_{i}\Omega^{j})-\sum_{i,j=1}^{3}\partial_{j}(\Omega^{i}\partial_{i}v^{j})\\
&=& \sum_{i,j=1}^{3}v^{i}\partial_{ij}\Omega^{j}+\sum_{i,j=1}^{3}\partial_{j}v^{i}\partial_{i}\Omega^{j}-\sum_{i,j=1}^{3}\Omega^{i}\partial_{ij}v^{j}-\sum_{i,j=1}^{3}\partial_{j}\Omega^{i}\partial_{i}v^{j}\\
&=& v\cdot\nabla\textnormal{div}\Omega+\sum_{i,j=1}^{3}\partial_{j}v^{i}\partial_{i}\Omega^{j}-\sum_{i,j=1}^{3}\partial_{j}\Omega^{i}\partial_{i}v^{j}-\Omega\cdot\nabla\textnormal{div}v\\
&=& v\cdot\nabla\textnormal{div}\Omega.
\end{eqnarray*}
Thus we obtain the equation
$$\partial_{t}\textnormal{div}\,\Omega+v\cdot\nabla \textnormal{div}\,\Omega=0.$$
Therefore the quantity $\textnormal{div}\,\Omega$ is transported by the flow and then $\Omega$ satisfies the condition of the incompressibility for every time.\\\\
$2)$ Denote $(\Omega^r,\Omega^\theta,\Omega^z)$ the coordinates of $\Omega$ in cylindrical basis, then we can write $\Omega$ under the form 
$$\Omega=\Omega^r e_r+\Omega^\theta e_\theta+\Omega^z e_z.$$ Now
\begin{eqnarray}\label{o}
\nonumber \Omega(t)\times e_\theta&=& (\Omega^r(t) e_r+\Omega^\theta(t) e_\theta+\Omega^z(t) e_z)\times e_\theta\\
 &=&\Omega^r(t)e_z-\Omega^z(t) e_r,
\end{eqnarray}
 where we have used $e_r\times e_\theta=e_z$, $e_\theta\times e_\theta=0$ and $e_z\times e_\theta=-e_r.$ Hence it suffices to prove that $\Omega^r(t)=\Omega^z(t)=0.$ For this purpose,  we write 
\begin{equation}\label{p}
\partial_{t}\Omega\cdot e_r+(v\cdot\nabla \Omega)\cdot e_r=(\Omega\cdot\nabla v)\cdot e_r+curl(\rho e_{z})\cdot e_r.
\end{equation}
It is clear that $\Omega^r=\Omega\cdot e_r.$  
Now from \eqref{f}, we have
$$v\cdot\nabla=v^r \partial_r+v^z \partial_z$$ and since $\partial_r e_r=\partial_z e_r=0,$ we can get,
\begin{eqnarray}\label{q}
\nonumber(v\cdot\nabla\Omega)\cdot e_r&=&(v^r \partial_r \Omega+v^z \partial_z \Omega)\cdot e_r\\
\nonumber &=&(v^r \partial_r+v^z \partial_z)(\Omega\cdot e_r)\\
&=&v\cdot\nabla\Omega^r.
\end{eqnarray}
And
\begin{eqnarray}\label{r}
\nonumber(\Omega\cdot\nabla v)\cdot e_r&=&\Omega^r \partial_r v\cdot e_r+\frac{1}{r}\Omega^\theta\partial_\theta v\cdot e_r+\Omega^z \partial_z v\cdot e_r\\
&=&\Omega^r \partial_r v^r+\Omega^z \partial_z v^r.
\end{eqnarray}
According to \eqref{aa} we immediately have, 
\begin{equation}\label{s}
curl(\rho e_z)\cdot e_r=0.
\end{equation}
Putting \eqref{q},\eqref{r} and \eqref{s} into \eqref{p} we obtain 
$$\partial_{t}\Omega^r+v\cdot\nabla\Omega^r=\Omega^r \partial_r v^r+\Omega^z \partial_z v^r.$$
The maximum modulus principle gives
\begin{equation}\label{t}
\Vert\Omega^r(t)\Vert_{L^\infty}\le\int^t_0\big(\Vert\Omega^r(\tau)\Vert_{L^\infty}+\Vert\Omega^z(\tau)\Vert_{L^\infty} \big)\Vert\nabla v(\tau)\Vert_{L^\infty}d\tau.
\end{equation}
Similarly for the component $\Omega^z$, we find the equation,
$$\partial_{t}\Omega^z+v\cdot\nabla\Omega^z=\Omega^z \partial_z v^z+\Omega^r \partial_r v^z.$$
This leads to 
\begin{equation}\label{u}
\Vert\Omega^z(t)\Vert_{L^\infty}\le\int^t_0\big(\Vert\Omega^z(\tau)\Vert_{L^\infty}+\Vert\Omega^r(\tau)\Vert_{L^\infty} \big)\Vert\nabla v(\tau)\Vert_{L^\infty}d\tau.
\end{equation}
From \eqref{t} and \eqref{u} we get
$$\Vert\Omega^r(t)\Vert_{L^\infty}+\Vert\Omega^z(t)\Vert_{L^\infty}\lesssim\int^t_0\big(\Vert\Omega^r(\tau) \Vert_{L^\infty} +\Vert\Omega^z(\tau)\Vert_{L^\infty}\big)\Vert\nabla v(\tau)\Vert_{L^\infty}d\tau.$$
Using Gronwall's inequality, we obtain for every $t\in\RR_+,$
\begin{equation}\label{v}
\Omega^r(t)=\Omega^z(t)=0.
\end{equation}
This implies that $\Omega=\Omega^{\theta} e_{\theta}.$\\
Plugging \eqref{v} into \eqref{o}, we obtain finally
$$\Omega(t)\times e_\theta=(0,0,0).$$
Now, we use two facts, the first is the axisymmetry of the vector field $v$ and the second is \eqref{v}, which give:
\begin{eqnarray*}
\Omega\cdot\nabla v &=&\Omega^r \partial_r v+\frac{1}{r}\Omega^\theta \partial_\theta v+\Omega^z \partial_z v\\
&=&\frac{1}{r}\Omega^\theta \partial_\theta v\\
&=&\frac{1}{r}\Omega^\theta \partial_\theta(v^r e_r+v^z e_z)\\
&=&\frac{1}{r}\Omega^\theta(v^r \partial_\theta e_r+v^z \partial_\theta e_z)\\
&=&\frac{1}{r}\Omega^\theta v^r e_\theta.
\end{eqnarray*}
Then we obtain
$$\Omega\cdot\nabla v=\frac{1}{r}v^r\Omega.$$
The proof of the Proposition is now complete.
\end{proof}
\section{Proof of main result}\label{D} 
In this section we will prove Theorem \ref{Theo1}. The main step for the proof is to give a global a priori estimates of the Lipschitz norm of the velocity which is the significant quantity to bound due to the general theory of hyperbolic systems. This allows us to propagate the initial Besov regularity. We will use the method of \cite{ahs08} for the proof.

\subsection{A priori estimates}
We will prove three kinds of a priori estimates : the first one deals with some easy estimates that one can obtained by energy estimates. The second one is concerned with a global a priori estimate of the Lipschitz norm of the velocity and the last a priori estimates concerned with some strong estimates. 
\subsubsection{Weak a priori estimates}
First we prove the the following energy estimate,
\begin{prop}\label{prop a1} Let $(v,\rho)$ be a smooth solution of the system \eqref{a}, then\\
$(a)$ for $\rho^{0}\in L^{2}\cap L^{p}$ with $6< p$ and $t \in \RR_+,$ we have
$$\Vert\rho(t)\Vert_{L^2 \cap L^p}\le\Vert\rho^0\Vert_{L^2 \cap L^p}.$$
$(b)$ For $\rho^{0}\in L^{2}$ and $\kappa=1$ we have,
$$\Vert\rho\Vert^{2}_{L_t^\infty L^2}+2\Vert\nabla\rho\Vert^{2}_{L_t^2 L^2}=\Vert\rho^0 \Vert^{2}_{L^2}.$$
$(c)$ for $v^0 \in L^2$, $\rho^0 \in L^2$ and $t \in \RR_+,$ we have
$$\Vert v(t)\Vert_{L^2}\le C_0 (1+t),$$
where $C_0$ depends only on $\Vert v^0 \Vert_{L^2}$ and $\Vert \rho^0 \Vert_{L^2}.$
\end{prop}
Notice that the axisymmetric assumption is not needed in this Proposition.
\begin{proof}
The first estimate is a direct consequence of Lemma \ref{m}. To prove $(b)$, we take the $L^2$scalar product of the second equation of \eqref{a} with $\rho$ and integrating by parts we get since $v$ is divergence free  
$$\frac{1}{2}\frac{d}{dt}\Vert\rho(t)\Vert^2_{L^2}+\int_{\RR^3}\vert\nabla\rho(t,x)\vert^{2}dx=0.$$
Integrating in time this differential inequality, we get
$$\Vert\rho\Vert^{2}_{L_t^\infty L^2}+2\Vert\nabla\rho\Vert^{2}_{L_t^2 L^2}=\Vert\rho^0 \Vert^{2}_{L^2}.$$
For$(c),$ we take the $L^2(\RR^3)$ scalar product of the velocity equation with $v.$ From integration by parts and the fact that $v$ is a divergence free, we get
\begin{equation}\label{ss1}
\frac{1}{2}\frac{d}{dt}\Vert v(t)\Vert^2_{L^2}\le \Vert v(t)\Vert_{L^2}\Vert \rho(t)\Vert_{L^2}.
\end{equation}
This yields
$$\frac{d}{dt}\Vert v(t)\Vert_{L^2}\le \Vert \rho(t)\Vert_{L^2}.$$ 
By integration in time, we find that
$$\Vert v(t)\Vert_{L^2}\le \Vert v^0 \Vert_{L^2}+\int_0^t \Vert \rho(\tau)\Vert_{L^2}d\tau.$$
Since $\Vert\rho(t)\Vert_{L^2}\le \Vert \rho^0 \Vert_{L^2},$ we infer  
$$\Vert v(t)\Vert_{L^2}\le \Vert v^0 \Vert_{L^2}+t \Vert \rho^0 \Vert_{L^2}.$$
\end{proof}
We have now some estimates for the moment of the density $\rho$
\begin{prop}\label{p1}Let $v$ be a vector field with zero divergence satisfying the energy estimate of Proposition \ref{prop a1} and let $\rho$ be a solution of the second equation of \eqref{a}. Then we have the folliowing estimates for the moment of $\rho.$\\
$a)$ For $\rho^{0}\in L^2$ and $x_{h}\rho^{0}\in L^{2},$ then there exists $C_{0}>0$ such that for every $t>0$
$$\Vert x_{h}\rho\Vert_{L_{t}^{\infty}L^2}+\Vert x_{h}\rho\Vert_{L_{t}^{2}\dot{H}^{1}}\lesssim C_{0}(1+t^{\frac{5}{4}}).$$
$b)$ For $\rho^{0}\in L^{2}$ and $\vert x_{h}\vert^{2}\rho^{0}\in L^{2},$ there exists $C_{0}>0$ such that for every $t\ge0$
$$\Vert\vert x_{h}\vert^{2}\rho\Vert_{L_{t}^{\infty}L^2}+\Vert\vert x_{h}\vert^{2}\rho\Vert_{L_{t}^{2}\dot{H}^{1}}\le C_{0}(1+         t^{\frac{5}{2}}).$$
\end{prop}
\begin{proof}
The estimate of $a)$ is proved by Hmidi and Rousset in \cite{tf010}, then for the estimate of $b)$ we first have 
the moment $\vert x_{h}\vert^{2}\rho$ solves the equation
\begin{equation}\label{66}
\partial_{t}(\vert x_{h}\vert^{2}\rho)+v\cdot\nabla(\vert x_{h}\vert^{2}\rho)-\Delta(\vert x_{h}\vert^{2}\rho)=2v^{h}(x_{h}\rho)-2\nabla_{h}\rho-4\textnormal{div}_{h}(x_{h}\rho),
\end{equation}
with $v^{h}=(v^1,v^2)$ and $\nabla_{h}=(\partial_1,\partial_2).$ Taking the $L^2$ inner product with $\vert x_{h}\vert^{2}\rho,$ integrating by parts and using H\"older inequality, 
\begin{eqnarray*}
\frac{1}{2}\frac{d}{dt}\Vert(\vert x_{h}\vert^{2}\rho)\Vert^{2}_{L^2}+\Vert\nabla(\vert x_{h}\vert^{2}\rho)\Vert^{2}_{L^2}&=& 2\int_{\RR^{3}} v^{h} (x_{h}\rho)(\vert x_{h}\vert^{2}\rho) dx-2\int_{\RR^{3}}\nabla_{h}\rho (\vert x_{h}\vert^{2}\rho) dx\\
&-& 4\int_{\RR^{3}}\textnormal{div}_{h}(x_{h}\rho)(\vert x_{h}\vert^{2}\rho) dx\\
&\le& 2\Vert v \Vert_{L^2}\Vert x_{h}\rho\Vert_{L^3}\Vert\vert x_{h}\vert^{2}\rho\Vert_{L^6}\\
&+& (2\Vert\rho\Vert_{L^2}+4\Vert x_{h}\rho\Vert_{L^2})\Vert\nabla(\vert x_{h}\vert^{2}\rho)\Vert_{L^2}.
\end{eqnarray*}
Using the embedding $\dot{H}^{1}\hookrightarrow L^{6}$ with Young inequality,
$$\frac{d}{dt}\Vert(\vert x_{h}\vert^{2}\rho)\Vert^{2}_{L^2}+c\Vert\nabla(\vert x_{h}\vert^{2}\rho)\Vert^{2}_{L^2}\lesssim\Vert v \Vert^{2}_{L^2}\Vert x_{h}\rho\Vert^{2}_{L^3}+\Vert\rho\Vert^{2}_{L^2}+\Vert x_{h}\rho\Vert^{2}_{L^2}.$$
Using now the Gagliardo-Nirenberg inequality,
$$\Vert f\Vert^{2}_{L^3}\le \Vert f \Vert_{L^2}\Vert\nabla f \Vert_{L^2},$$ 
we get that,
$$\frac{d}{dt}\Vert(\vert x_{h}\vert^{2}\rho)\Vert^{2}_{L^2}+c\Vert\nabla(\vert x_{h}\vert^{2}\rho)\Vert^{2}_{L^2}\lesssim \Vert v \Vert^{2}_{L^2}\Vert x_{h}\rho\Vert_{L^2}\Vert\nabla(x_{h}\rho)\Vert_{L^2}+\Vert\rho\Vert^{2}_{L^2}+\Vert x_{h}\rho\Vert^{2}_{L^2}.$$
Integrating in time and using the estimate $a)$ for $x_{h}\rho$ and Proposition \ref{prop a1}, we obtain
\begin{eqnarray*}
\Vert(\vert x_{h}\vert^{2}\rho)\Vert^{2}_{L^2}+c\int_{0}^{t}\Vert\nabla(\vert x_{h}\vert^{2}\rho)\Vert^{2}_{L^2}d\tau&\lesssim& \Vert(\vert x_{h}\vert^{2}\rho^{0})\Vert^{2}_{L^2}+\Vert v \Vert^{2}_{L_{t}^{\infty}L^2}\Vert x_{h}\rho\Vert_{L_{t}^{\infty}L^2}t^{\frac{1}{2}}\Vert\nabla(x_{h}\rho)\Vert_{L_{t}^{2}L^2}\\
&+& \Vert\rho^{0}\Vert^{2}_{L^2}t+\Vert x_{h}\rho\Vert^{2}_{L_{t}^{\infty}L^2}t\\
&\lesssim& C_{0}(1+t^{5}).
\end{eqnarray*}
Then we obtain
$$\Vert\vert x_{h}\vert^{2}\rho\Vert_{L_{t}^{\infty}L^2}+\Vert\vert x_{h}\vert^{2}\rho\Vert_{L_{t}^{2}\dot{H}^{1}}\lesssim C_{0}(1+t^{\frac{5}{2}}).$$
This ends the proof of the proposition.
\end{proof}
We give now an a priori bound of the $L^\infty$-norm of $\frac{v^r}{r}$, see Proposition 4.4 in \cite{tf010} for the proof. 
\begin{prop}\label{prop a2} Let $v^0$ be a smooth axisymmetric vector field with zero divergence such that $v^0\in L^2$, its vorticity such that $\frac{\omega^0}{r}\in L^{3,1}$ and $\rho^0 \in L^{2}\cap L^{p}$ for $6<p$ axisymmetric such that $\vert x_{h}\vert^{2}\rho^{0}\in L^2$ with $x_h=(x_1,x_2).$ Then we have for every $t \in \RR_+$ 
$$\Vert\frac{\omega(t)}{r}\Vert_{L^{3,1}}+\Vert \frac{v^r(t)}{r}\Vert_{L^{\infty}} \le \Phi_{2}(t),$$
where $C_0$ is a constant depending on the norm of the initial data and recall that
$$\Phi_{l}(t)=C_{0}\underbrace{\exp(...\exp}_{l-times}(C_{0} t^{\frac{19}{6}})...).$$
\end{prop}
Let us now show how to use this to give an $L^\infty$- bound of the vorticity,
\begin{prop}\label{prop a3} Under the same assumptions of Proposition \ref{prop a2}, and if in addition $\omega^0 \in L^{\infty}.$ Then we have for every $t \in \RR_+,$
$$\Vert\omega(t)\Vert_{L^\infty}+\Vert\nabla\rho(t)\Vert_{L_{t}^{1} L^\infty}\lesssim \Phi_{4}(t).$$
\end{prop}
\begin{proof} Recall that the vorticity $\omega$ satisfies the equation
$$\partial_{t}\omega+v\cdot\nabla\omega=\frac {v^r}{r}\omega+curl(\rho e_{z}).$$
 Applying the maximum principle to the above equation and using Proposition \ref{prop a2},
\begin{eqnarray*}
\Vert\omega(t)\Vert_{L^\infty}&\le& \Vert\omega^0\Vert_{L^\infty}+\int_0^t \Vert \frac{v^r(\tau)}{r}\Vert_{L^{\infty}}\Vert\omega(\tau)\Vert_{L^\infty} d\tau+ \int_0^t \Vert curl(\rho e_z)(\tau) \Vert_{L^\infty}d\tau\\
&\le&  \Vert\omega^0\Vert_{L^\infty}+ \int_{0}^{t} \Phi_{2}(\tau) \Vert\omega(\tau)\Vert_{L^\infty} d\tau+ \int_0^t \Vert\nabla\rho(\tau)\Vert_{L^\infty} d\tau. 
\end{eqnarray*}
This implies by Gronwall inequality,
\begin{eqnarray}\label{rf}
\nonumber \Vert\omega(t)\Vert_{L^\infty}&\le& \bigg(\Vert\omega^0\Vert_{L^\infty}+ \int_0^t \Vert\nabla\rho(\tau)\Vert_{L^\infty} d\tau\bigg)\Phi_{3}(t)\\
&\lesssim& \bigg(\Vert\omega^0\Vert_{L^\infty}+ \int_0^t \Vert\rho(\tau)\Vert_{B^{1}_{\infty,1}} d\tau\bigg)\Phi_{3}(t),
\end{eqnarray}
where we have used in the last line the Besov embedding $B^{1}_{\infty,1}\hookrightarrow Lip.$ It remain then to estimate $\Vert\rho(t)\Vert_{L_t^1B^{1}_{\infty,1}}.$ For this purpose we use Bernstein inequality for $p>3$ and Propositions \ref{prop a1} and \ref{prop3}, we obtain
\begin{eqnarray}\label{x1}
\nonumber \Vert\rho(t)\Vert_{L_t^1B^{1}_{\infty,1}}&\lesssim& \Vert\Delta_{-1}\rho(t)\Vert_{L_t^1L^\infty}+\sum_{j\ge 0}2^{j}\Vert \Delta_{j}\rho(t)\Vert_{L_t^1L^\infty}\\
\nonumber &\lesssim&  \Vert\rho(t)\Vert_{L_t^1L^2}+\sum_{j\ge0}2^{j(\frac{3}{p}+1)}\Vert \Delta_{j}\rho(t)\Vert_{L_t^1L^p}\\
\nonumber &\lesssim&  \Vert\rho^{0}\Vert_{L^2}t+\sum_{j\ge0}2^{j(\frac{3}{p}-1)}\Vert\rho^0 \Vert_{L^p}\bigg(1+(j+1) \int_0^t \Vert\omega(\tau)\Vert_{L^\infty} d\tau\bigg)\\
\nonumber &+& \sum_{j\ge0}2^{j(\frac{3}{p}-1)}\Vert\rho^0 \Vert_{L^p} \int_0^t \Vert\nabla\Delta_{-1} v(\tau)\Vert_{L^\infty} d\tau\\
\nonumber &\lesssim& \Vert\rho^{0}\Vert_{L^2}t+\Vert\rho^0 \Vert_{L^p}\Big(1+\int_0^t \Vert\omega(\tau)\Vert_{L^\infty} d\tau+\int_0^t \Vert v(\tau)\Vert_{L^2} d\tau\Big)\\
\nonumber &\lesssim& \Vert\rho^{0}\Vert_{L^2}t+\Vert\rho^0 \Vert_{L^p}\Big(1+C_{0}t(1+t)+\int_0^t \Vert\omega(\tau)\Vert_{L^\infty} d\tau\Big)\\
&\lesssim& C_{0}\Big(1+t^{2}+\int_0^t \Vert\omega(\tau)\Vert_{L^\infty}d\tau\Big).
\end{eqnarray}
Plugging \eqref{x1} into \eqref{rf} and using Gronwall's inequality, we obtain
\begin{eqnarray*}
\Vert\omega(t)\Vert_{L^\infty}&\lesssim& \bigg(\Vert\omega^0\Vert_{L^\infty}+C_0 (1+t^2+\int_0^t \Vert\omega(\tau)\Vert_{L^\infty} d\tau)\bigg)\Phi_{3}(t)\\
&\lesssim& \Phi_{4}(t).
\end{eqnarray*}
Putting this estimate in \eqref{x1} yields
\begin{equation}\label{dh}
\Vert\rho(t)\Vert_{L_{t}^{1}B^{1}_{\infty,1}}\lesssim \Phi_{4}(t).
\end{equation}
This gives in view of Besov embedding,
$$\Vert\nabla\rho(t)\Vert_{L_{t}^{1} L^\infty}\lesssim \Phi_{4}(t),$$
which is the desired result.
\end{proof}
\subsubsection{Lipschitz estimate of the velocity} We will prove in the following proposition a new decomposition of the vorticity which is the basic tool to get a bound of the velocity.
\begin{prop}\label{prop a4} There exists a decomposition $(\widetilde{\omega}_{j})_{j\ge-1}$ of $\omega$ such that 
\begin{enumerate}
\item $\omega(t,x)=\displaystyle \sum_{j \ge -1}\widetilde{\omega}_{j}(t,x)\;,\,\forall t \in\RR_+.$\\
\item $\textnormal{div}\;\widetilde{\omega}_{j}(t,x)=0.$\\
\item $\Vert\widetilde{\omega}_{j}(t)\Vert_{L^\infty}\le C (\Vert\Delta_{j}\omega^{0}\Vert_{L^\infty}+2^{j}\Vert\Delta_{j}\rho(t)\Vert_{L_{t}^{1}L^\infty})\Phi_{3}(t)\;,\;\forall j\ge -1.$\\
\item $\forall k, j \ge-1$ we have
$$\Vert\Delta_{k}\widetilde{\omega}_{j}(t)\Vert_{L^\infty}\lesssim 2^{-\vert k-j \vert}e^{C V(t)}(\Vert\Delta_{j}\omega^{0}\Vert_{L^\infty}+ 2^{j}\Vert\Delta_{j}\rho(t)\Vert_{L_{t}^{1}L^\infty}),$$
with $V(t):=\Vert v(t) \Vert_{L_t^1 B_{\infty,1}^{1}}.$
\end{enumerate}
\end{prop}
\begin{proof}
We will use for this purpose a new approach similar to \cite{ahs08}. We denote by $(\widetilde{\omega}_{j})_{j\ge-1}$ the unique global solution of the following equation,
\begin{equation}\label{rt}
\left\{
\begin{array}{ll} 
\partial_{t}\widetilde{\omega}_{j}+v\cdot\nabla\widetilde{\omega}_{j}=\widetilde{\omega}_{j}\cdot\nabla v+curl(\Delta_j\rho e_{z})\\ 
\widetilde{\omega}_{j}(t=0)=\Delta_j\omega^{0}
\end{array} \right.
\end{equation} 
Since $\textnormal{div}(\Delta_{j} \omega^0)=0$ then Proposition \ref{prop5}-1) implies that $$\textnormal{div}\,\widetilde{\omega}_{j}(t)=0.$$ 
By linearity and uniqueness we have
\begin{equation*}
\omega(t,x)=\sum_{j\ge-1}\widetilde{\omega}_{j}(t,x),
\end{equation*}
which is (1).\\
To prove (3) we apply Proposition \ref{prop5}-2) to the equation \eqref{rt}, since $\Delta_j\omega^0=curl \Delta_jv^0$ and $\Delta_jv^0$ is axisymmetric then Proposition \ref{prop4} gives $\Delta_j\omega^0\times e_\theta=(0,0,0)$, therefore $\widetilde{\omega}_{j}(t)\times e_\theta=(0,0,0)$ and
$$\partial_{t}\widetilde{\omega}_{j}+v\cdot\nabla\widetilde{\omega}_{j}= \widetilde{\omega}_{j}\frac{v^r}{r}+curl(\Delta_j\rho e_{z})$$ 
Now taking the maximum principle and using Propositions \ref{prop a2} and Lemma \ref{Bernstein},
\begin{eqnarray*}
\Vert\widetilde{\omega}_{j}(t)\Vert_{L^\infty}&\le&\Vert\Delta_j\omega^0\Vert_{L^\infty}+\int_0^t \Vert\frac{v^r(\tau)}{r}\Vert_{L^\infty}\Vert\widetilde{\omega}_{j}(\tau)\Vert_{L^\infty}d\tau+\int_0^t \Vert curl(\Delta_j\rho e_{z})(\tau)\Vert_{L^\infty}d\tau\\
&\lesssim& \Vert\Delta_j\omega^0\Vert_{L^\infty}+ \int_0^t \Phi_{2}(\tau)\Vert\widetilde{\omega}_{j}(\tau)\Vert_{L^\infty}d\tau+2^{j}\int_0^t \Vert \Delta_j\rho(\tau)\Vert_{L^\infty}d\tau.
\end{eqnarray*} 
Using Gronwall's inequality we obtain
\begin{equation*}
\Vert\widetilde{\omega}_{j}(t)\Vert_{L^\infty}\le C(\Vert\Delta_j\omega^0\Vert_{L^\infty}+2^{j}\Vert \Delta_j\rho(t)\Vert_{L_{t}^{1}L^\infty})\Phi_{3}(t).
\end{equation*}
This is the result (3).\\
Let us now prove (4). Notice that this estimate is equivalent to
\begin{equation}\label{t1}
\Vert\Delta_{k}\widetilde{\omega}_{j}(t)\Vert_{L^\infty}\lesssim 2^{k-j}e^{C V(t)}(\Vert\Delta_{j}\omega^{0}\Vert_{L^\infty} + 2^{j}\Vert \Delta_j\rho(t)\Vert_{L_{t}^{1}L^\infty})
\end{equation}
and
\begin{equation}\label{t2}
\Vert\Delta_{k}\widetilde{\omega}_{j}(t)\Vert_{L^\infty}\lesssim 2^{j-k}e^{C V(t)}(\Vert\Delta_{j}\omega^{0}\Vert_{L^\infty} + 2^{j}\Vert \Delta_j\rho(t)\Vert_{L_{t}^{1}L^\infty}).
\end{equation}
\underline{Proof of \eqref{t1}}. Applying Proposition \ref{prop2} to the equation \eqref{rt},
\begin{eqnarray}\label{y}
\nonumber e^{-CV(t)}\Vert\widetilde{\omega}_{j}(t)\Vert_{B^{-1}_{\infty,\infty}}&\lesssim& \Vert\Delta_{j}\omega^0\Vert_{B^{-1}_{\infty,\infty}}+ \int_0^t e^{-CV(\tau)}\Vert\widetilde{\omega}_{j}\cdot\nabla v(\tau)\Vert_{B^{-1}_{\infty,\infty}}d\tau\\
&+& \int_0^t e^{-CV(\tau)}\Vert curl (\Delta_{j}\rho e_{z})\Vert_{B^{-1}_{\infty,\infty}}d\tau,
\end{eqnarray}
where $V(t)=\int_{0}^{t}\Vert v(\tau)\Vert_{B^{1}_{\infty,1}}d\tau.$\\
To estimate the first integral, we use the decomposition of Bony \eqref{j}, then 
\begin{eqnarray*}
\Vert\widetilde{\omega}_{j}\cdot\nabla v\Vert_{B^{-1}_{\infty,\infty}}&\le& \Vert T_{\widetilde{\omega}_{j}}\cdot\nabla v\Vert_{B^{-1}_{\infty,\infty}}+\Vert T_{\nabla v}\cdot\widetilde{\omega}_{j}\Vert_{B^{-1}_{\infty,\infty}}\\
&+& \Vert \mathcal{R}(\widetilde{\omega}_{j},\nabla v)\Vert_{B^{-1}_{\infty,\infty}}.
\end{eqnarray*}
Now we write by definition
$$T_{\widetilde{\omega}_{j}}\cdot\nabla v=\sum_{q} S_{q-1}\widetilde{\omega}_{j}\Delta_{q}\nabla v,$$
then 
$$\Delta_{k}(T_{\widetilde{\omega}_{j}}\cdot\nabla v)=\sum_{\vert k-q \vert \le 4}\Delta_{k}(S_{q-1}\widetilde{\omega}_{j}\Delta_{q}\nabla v).$$
Applying H\"older and Bernstein inequalities leads to
\begin{eqnarray*}
\Vert T_{\widetilde{\omega}_{j}}\cdot\nabla v\Vert_{B^{-1}_{\infty,\infty}}&\le& \sup_{k} 2^{-k} \sum_{\vert k-q \vert \le 4}\Vert S_{q-1}\widetilde{\omega}_{j}\Vert_{L^\infty}\Vert\Delta_{q}\nabla v\Vert_{L^\infty}\\
&\lesssim& \Vert\nabla v\Vert_{L^\infty} \Vert\widetilde{\omega}_{j}\Vert_{B^{-1}_{\infty,\infty}}\\
&\lesssim& \Vert v \Vert_{B_{\infty,1}^1} \Vert\widetilde{\omega}_{j}\Vert_{B^{-1}_{\infty,\infty}},
\end{eqnarray*}
we have used in the last line the Besov emmbedding $B_{\infty,1}^{1} \hookrightarrow Lip(\RR^3).$\\
Similarly for $\Vert T_{\nabla v}\cdot\widetilde{\omega}_{j}\Vert_{B^{-1}_{\infty,\infty}},$ we get
$$\Vert T_{\nabla v}\cdot\widetilde{\omega}_{j}\Vert_{B^{-1}_{\infty,\infty}}\lesssim \Vert v \Vert_{B_{\infty,1}^1} \Vert\widetilde{\omega}_{j}\Vert_{B^{-1}_{\infty,\infty}}.$$
For the remainder term, since $\textnormal{div}\,\widetilde{\omega}_{j}=0$ we have
\begin{eqnarray*}
\Vert \mathcal{R}(\widetilde{\omega}_{j},\nabla v)\Vert_{B^{-1}_{\infty,\infty}}&=& \Vert \textnormal{div}\mathcal{R}(\widetilde{\omega}_{j}\otimes v)\Vert_{B^{-1}_{\infty,\infty}}\\
&\lesssim& \sup_{k} \sum_{q\ge k -3}\Vert\Delta_{q}\widetilde{\omega}_{j}\Vert_{L^\infty}\Vert\widetilde{\Delta}_{q} v \Vert_{L^\infty}\\
&\lesssim& \Vert v \Vert_{B_{\infty,1}^1} \Vert\widetilde{\omega}_{j}\Vert_{B^{-1}_{\infty,\infty}}.
\end{eqnarray*} 
Therefore,
\begin{equation}\label{z}
\Vert\widetilde{\omega}_{j}\cdot\nabla v\Vert_{B^{-1}_{\infty,\infty}}\lesssim \Vert v \Vert_{B_{\infty,1}^1} \Vert\widetilde{\omega}_{j}\Vert_{B^{-1}_{\infty,\infty}}.
\end{equation}
Using now H\"older and Bernstein inequalities for the estimate of the second integral of \eqref{y}, we obtain
\begin{eqnarray*}
\int_0^t e^{-CV(\tau)}\Vert curl (\Delta_{j}\rho e_{z})\Vert_{B^{-1}_{\infty,\infty}}d\tau &\le& \Vert curl (\Delta_{j}\rho)\Vert_{L_{t}^{1}B^{-1}_{\infty,\infty}} \\
&\lesssim& \Vert \Delta_{j}\rho\Vert_{L_{t}^{1} B^{0}_{\infty,\infty}}\\
&\lesssim& \Vert \Delta_{j}\rho\Vert_{L_{t}^{1} L^\infty}.
\end{eqnarray*}
Hence it follows that into \eqref{y},
\begin{equation*}
e^{-CV(t)}\Vert\widetilde{\omega}_{j}(t)\Vert_{B^{-1}_{\infty,\infty}}\lesssim \Vert\Delta_{j}\omega^0\Vert_{B^{-1}_{\infty,\infty}}+\Vert \Delta_{j}\rho\Vert_{L_{t}^{1} L^\infty}+\int_0^t e^{-CV(\tau)}\Vert\widetilde{\omega}_{j}\Vert_{B^{-1}_{\infty,\infty}}\Vert v(\tau)\Vert_{B_{\infty,1}^1} d\tau.
\end{equation*}
According to the Gronwall inequality, we get
\begin{equation*}
e^{-CV(t)}\Vert\widetilde{\omega}_{j}(t)\Vert_{B^{-1}_{\infty,\infty}}\lesssim \Big(\Vert\Delta_{j}\omega^0\Vert_{B^{-1}_{\infty,\infty}}+\Vert \Delta_{j}\rho\Vert_{L_{t}^{1} L^\infty}\Big)e^{C V(t)}.
\end{equation*}
This implies that
\begin{eqnarray*}
\Vert\widetilde{\omega}_{j}(t)\Vert_{B^{-1}_{\infty,\infty}}&\lesssim& (\Vert\Delta_{j}\omega^0\Vert_{B^{-1}_{\infty,\infty}}+\Vert \Delta_{j}\rho\Vert_{L_{t}^{1} L^\infty}) e^{CV(t)}\\
&\lesssim& (2^{-j} \Vert\Delta_{j}\omega^0\Vert_{L^{\infty}}+\Vert \Delta_{j}\rho(t)\Vert_{L_{t}^{1} L^\infty}) e^{CV(t)}. 
\end{eqnarray*}
Finally we get by definition,
\begin{equation*}
\Vert\Delta_{k}\widetilde{\omega}_{j}(t)\Vert_{L^{\infty}}\lesssim 2^{k-j} e^{CV(t)}(\Vert\Delta_{j}\omega^0\Vert_{L^{\infty}} +2^{j} \Vert \Delta_{j}\rho(t)\Vert_{L_{t}^{1} L^\infty}).
\end{equation*}
which is the desired result.\\
\underline{Proof of \eqref{t2}}. Since the $z$-component of $\omega^0$ is zero, then $\widetilde{\omega}_{j}=(\widetilde{\omega}_{j}^1,\widetilde{\omega}_{j}^2,0).$ We are going to work with the two components separately. The analysis will be exactly the same so we deal only with first component $\widetilde{\omega}_{j}^{1}.$ From the identity $\frac{v^r}{r}=\frac{v^1}{x_1}=\frac{v^2}{x_2}$ which is an easy consequence of $v^\theta=0,$ it is clear that the functions $\widetilde{\omega}_{j}^1$ solves the equation,
\begin{equation*} 
\left\{\begin{array}{ll} 
\partial_{t}\widetilde{\omega}_{j}^1+v\cdot\nabla\widetilde{\omega}_{j}^1= v^2\frac{\widetilde{\omega}_{j}^1}{x_2}+\partial_{2}\Delta_j\rho\\
\widetilde{\omega}_{j}^{1}(t=0)=\Delta_{j}\omega_{0}^{1}.  
\end{array} \right.
\end{equation*} 
Applying Proposition \ref{prop2} to the above equation,
\begin{eqnarray}\label{z1}
\nonumber e^{-CV(t)}\Vert\widetilde{\omega}_{j}^1(t)\Vert_{B^{1}_{\infty,1}}&\lesssim& \Vert\Delta_{j}\omega^1_0\Vert_{B^{1}_{\infty,1}}+ \int_0^t e^{-CV(\tau)}\Vert v^2\frac{\widetilde{\omega}_{j}^1(\tau)}{x_2}\Vert_{B^{1}_{\infty,1}}d\tau\\
&+& \int_0^t e^{-CV(\tau)}\Vert\partial_{2}\Delta_{j}\rho(\tau)\Vert_{B^{1}_{\infty,1}}d\tau,
\end{eqnarray}
To estimate the first integral, we use the decomposition of Bony \eqref{j}, then we have
\begin{equation}\label{z2}
\Vert v^2\frac{\widetilde{\omega}_{j}^1}{x_2}\Vert_{B^{1}_{\infty,1}}\le \Vert T_{v^2}\frac{\widetilde{\omega}_{j}^1}{x_2}\Vert_{B^{1}_{\infty,1}}+\Vert T_{\frac{\widetilde{\omega}_{j}^1}{x_2}}v^2 \Vert_{B^{1}_{\infty,1}}+\Vert\mathcal{R}(v^2,\frac{\widetilde{\omega}_{j}^1}{x_2})\Vert_{B^{1}_{\infty,1}}. 
\end{equation}
We have by definition,
\begin{eqnarray}\label{z3}
\nonumber \Vert T_{\frac{\widetilde{\omega}_{j}^1}{x_2}}v^2 \Vert_{B^{1}_{\infty,1}}&\lesssim& \sum_{k} 2^{k}\Vert S_{k-1}(\frac{\widetilde{\omega}_{j}^1}{x_2})\Vert_{L^{\infty}}\Vert\Delta_{k}v^{2}\Vert_{L^\infty}\\
\nonumber &\lesssim& \bigg\Vert\frac{\widetilde{\omega}_{j}^1}{x_2}\bigg\Vert_{L^{\infty}}\Vert v \Vert_{B^{1}_{\infty,1}}\\ 
&\lesssim& \Big\Vert\frac{\widetilde{\omega}_{j}^1}{x_2}\Big\Vert_{B^{0}_{\infty,1}}\Vert v \Vert_{B^{1}_{\infty,1}}. 
\end{eqnarray}
The remainder term is estimated as follows,
$$\Vert\mathcal{R}(v^2,\frac{\widetilde{\omega}_{j}^1}{x_2})\Vert_{B^{1}_{\infty,1}}=\sum_{j}2^{j}\Vert\Delta_{j}\mathcal{R}(v^2,\frac{\widetilde{\omega}_{j}^1}{x_2})\Vert_{L^{\infty}}.$$
Now 
\begin{eqnarray*}
\Delta_{j}\mathcal{R}(v^2,\frac{\widetilde{\omega}_{j}^1}{x_2})&=& \Delta_{j}\sum_{k\ge -1}\Delta_{k}v^{2}\widetilde{\Delta}_{k}(\frac{\widetilde{\omega}_{j}^1}{x_2})\\
&=& \sum_{k}\Delta_{j}\big(\Delta_{k}v^{2}\widetilde{\Delta}_{k}(\frac{\widetilde{\omega}_{j}^1}{x_2})\big)
\end{eqnarray*}
Since $supp\mathcal{F}(\Delta_{k}v^{2}\widetilde{\Delta}_{k}(\frac{\widetilde{\omega}_{j}^1}{x_2}))\subset 2^{k}\mathcal{B},$ then $\Delta_{j}\big(\Delta_{k}v^{2}\widetilde{\Delta}_{k}(\frac{\widetilde{\omega}_{j}^1}{x_2})\big)=0$ if $k\le j-4.$ It follows that
$$\Vert\Delta_{j}\mathcal{R}(v^2,\frac{\widetilde{\omega}_{j}^1}{x_2})\Vert_{L^\infty}\lesssim\sum_{k\ge j-4}\Vert\Delta_{k}v^{2}\Vert_{L^\infty} \Vert\widetilde{\Delta}_{k}(\frac{\widetilde{\omega}_{j}^1}{x_2})\Vert_{L^{\infty}},$$
where we have used the continuity of the operator $\Delta_{j}$ in $L^\infty.$ Therefore
\begin{eqnarray}\label{z4}
\nonumber \Vert\mathcal{R}(v^2,\frac{\widetilde{\omega}_{j}^1}{x_2})\Vert_{B^{1}_{\infty,1}}&\lesssim& \sum_{j}2^{j}\sum_{k\ge j-4}\Vert\Delta_{k}v^{2}\Vert_{L^\infty}\Vert\widetilde{\Delta}_{k}(\frac{\widetilde{\omega}_{j}^1}{x_2})\Vert_{L^{\infty}}\\
\nonumber &\lesssim& \Big\Vert\frac{\widetilde{\omega}_{j}^1}{x_2}\Big\Vert_{L^{\infty}}\sum_{k}2^{k}\Vert\Delta_{k}v^{2}\Vert_{L^\infty}\sum_{j\le k+4}2^{j-k}\\
\nonumber &\lesssim& \Big\Vert\frac{\widetilde{\omega}_{j}^1}{x_2}\Big\Vert_{L^{\infty}}\Vert v \Vert_{B^{1}_{\infty,1}}\\ 
&\lesssim& \Big\Vert\frac{\widetilde{\omega}_{j}^1}{x_2}\Big\Vert_{B^{0}_{\infty,1}}\Vert v \Vert_{B^{1}_{\infty,1}}.
\end{eqnarray}
Now to estimate the term $\Vert T_{v^2}\frac{\widetilde{\omega}_{j}^1}{x_2}\Vert_{B^{1}_{\infty,1}},$ we use the axisymmetric structure of the vector field $v.$ We have by definition
\begin{equation*}
\Vert T_{v^2}\frac{\widetilde{\omega}_{j}^1}{x_2}\Vert_{B^{1}_{\infty,1}}\lesssim \sum_{k \ge 0} 2^{k}\Vert S_{k-1}v^{2}(x)\Delta_{k}(\frac{\widetilde{\omega}^{1}_{j}(x)}{x_2})\Vert_{L^{\infty}}.
\end{equation*}
We write now,
\begin{eqnarray*}
S_{k-1}v^{2}(x)\Delta_{k}(\frac{\widetilde{\omega}^{1}_{j}(x)}{x_2})&=& S_{k-1}v^{2}(x)\frac{\Delta_{k}\widetilde{\omega}^{1}_{j}(x)}{x_2}+S_{k-1}v^{2}(x)[\Delta_{k},\frac{1}{x_{2}}]\widetilde{\omega}^{1}_{j}\\
&:=&\textnormal{I}_{k}+\textnormal{II}_{k},
\end{eqnarray*}
where we have used the notation $$[\Delta_{j},a]b=\Delta_{j}(ab)-a\Delta_{j}b.$$
By Proposition \ref{prop4} we have $S_{k-1}v$ is axisymmetric and then $S_{k-1}v^{2}(x_1,0,z)=0.$ Therefore from Taylor formula,
\begin{equation*}
S_{k-1} v^{2}(x_1, x_2, z)=x_{2}\int_{0}^{1}\Big(\partial_{2} S_{k-1} v^{2}\Big)(x_1, \tau x_2, z) d\tau.
\end{equation*}
Then
\begin{equation}\label{z5}
\Vert \frac{S_{k-1} v^{2}(x)}{x_2}\Vert_{L^\infty}\lesssim \Vert\nabla v \Vert_{L^\infty}. 
\end{equation}
Thus
$$\Vert\textnormal{I}_{k}\Vert_{L^\infty}\lesssim \Vert\nabla v \Vert_{L^\infty}\Vert\Delta_{k}\widetilde{\omega}^{1}_{j}\Vert_{L^\infty}.$$
Therefore
\begin{eqnarray}\label{z6}
\nonumber \sum_{k \ge 0}2^{k}\Vert\textnormal{I}_{k}\Vert_{L^\infty}&\lesssim& \Vert\nabla v \Vert_{L^\infty}\sum_{k \ge 0}2^{k}\Vert\Delta_{k}\widetilde{\omega}^{1}_{j}\Vert_{L^\infty}\\
&\lesssim& \Vert\nabla v \Vert_{L^\infty}\Vert\widetilde{\omega}^{1}_{j}\Vert_{B^{1}_{\infty,1}} 
\end{eqnarray}
For the commutator term,we write by definition
\begin{eqnarray*} 
\textnormal{II}_{k}&=&S_{k-1}v^{2}(x)\Delta_{k}(\frac{\widetilde{\omega}^{1}_{j}(x)}{x_2})-\frac{S_{k-1}v^{2}(x)}{x_2}\Delta_{k}\widetilde{\omega}^{1}_{j}\\
&=& \frac{S_{k-1}v^{2}(x)}{x_2} 2^{3k}\int_{\RR^3}h(2^{k}(x-y))(x_2-y_2) \frac{\widetilde{\omega}^{1}_{j}}{y_2}(y) dy\\
&=& 2^{-k}\frac{S_{k-1}v^{2}(x)}{x_2}2^{3k}\widetilde{h}(2^{k}\cdot)\ast\Big(\frac{\widetilde{\omega}^{1}_{j}}{y_2} \Big)(x),
\end{eqnarray*}
where $\widetilde{h}(x)=x_2 h(x).$ Since $\mathcal{F}(\widetilde{h}(\xi))=i\partial_{\xi_{2}}\mathcal{F}(h(\xi))=i\partial_{\xi_{2}}\varphi(\xi).$ Then it follows that $supp\;\mathcal{F}(\widetilde{h})\subset supp\,\mathcal{F}(h)=supp\,\varphi.$ Therefore for every $g \in \mathcal{S}^\prime$ we have $2^{3k}\widetilde{h}(2^{k}\cdot)\ast\Delta_{p}g=0\;,\;\textnormal{for}\;\vert k-p\vert\ge2.$ This leads to
$$2^{3k}\widetilde{h}(2^{k}\cdot)\ast g=\sum_{\vert k-p \vert \le 1}2^{3k}\widetilde{h}(2^{k}\cdot)\ast \Delta_{p} g.$$
Hence by Young inequality for convolution and \eqref{z5} we get
\begin{eqnarray}\label{z7}
\nonumber \sum_{k \ge 0}2^{k}\Vert\textnormal{II}_{k}\Vert_{L^\infty}&\lesssim& \sum_{\vert k-p\vert \le 1}\Vert\frac{S_{k-1}v^{2}(x)}{x_2}\Vert_{L^\infty}\Vert 2^{3k}\widetilde{h}(2^{k}\cdot)\ast\Delta_{p}(\frac{\widetilde{\omega}^{1}_{j}}{x_2})\Vert_{L^\infty}\\
\nonumber &\lesssim& \Vert\nabla v \Vert_{L^\infty}\sum_{\vert k-p\vert \le 1}\Vert\widetilde{h}\Vert_{L^1}\Vert\Delta_{p}(\frac{\widetilde{\omega}^{1}_{j}}{x_2})\Vert_{L^\infty}\\
&\lesssim& \Vert\nabla v \Vert_{L^\infty}\Vert\frac{\widetilde{\omega}^{1}_{j}}{x_2}\Vert_{B^{0}_{\infty,1}}. 
\end{eqnarray}
Thus it follows from \eqref{z6} and \eqref{z7} that,
\begin{eqnarray}\label{z8}
\nonumber \Vert T_{v^2}\frac{\widetilde{\omega}_{j}^1}{x_2}\Vert_{B^{1}_{\infty,1}}&\lesssim& \Vert\nabla v \Vert_{L^\infty}(\Vert\widetilde{\omega}^{1}_{j}\Vert_{B^{1}_{\infty,1}}+\Vert\frac{\widetilde{\omega}^{1}_{j}}{x_2}\Vert_{B^{0}_{\infty,1}})\\
&\lesssim& \Vert v \Vert_{B^{1}_{\infty,1}}(\Vert\widetilde{\omega}^{1}_{j}\Vert_{B^{1}_{\infty,1}}+\Vert\frac{\widetilde{\omega}^{1}_{j}}{x_2}\Vert_{B^{0}_{\infty,1}}).
\end{eqnarray}
Now putting  together \eqref{z3}, \eqref{z4}, \eqref{z8} and \eqref{z2} we find
\begin{equation*}
\Vert v^2\frac{\widetilde{\omega}_{j}^1}{x_2}\Vert_{B^{1}_{\infty,1}}\lesssim \Vert v \Vert_{B^{1}_{\infty,1}}(\Vert\widetilde{\omega}^{1}_{j}\Vert_{B^{1}_{\infty,1}}+\Vert\frac{\widetilde{\omega}^{1}_{j}}{x_2}\Vert_{B^{0}_{\infty,1}}). 
\end{equation*}
Plugging this last estimate into \eqref{z1} and using Bernstein and H\"older inequalities, we find
\begin{eqnarray*}
e^{-CV(t)}\Vert\widetilde{\omega}_{j}^1(t)\Vert_{B^{1}_{\infty,1}}&\lesssim& \Vert\Delta_{j}\omega^1_0\Vert_{B^{1}_{\infty,1}}+\int_0^t e^{-CV(\tau)}\Vert\widetilde{\omega}_{j}^1(\tau)\Vert_{B^{1}_{\infty,1}}\Vert v(\tau)\Vert_{B^{1}_{\infty,1}}d\tau\\
&+& \int_0^t e^{-CV(\tau)}\Vert\frac{\widetilde{\omega}^{1}_{j}}{x_2}\Vert_{B^{0}_{\infty,1}}\Vert v(\tau)\Vert_{B^{1}_{\infty,1}}d\tau+2^{2j}\Vert\Delta_{j}\rho(t)\Vert_{L_{t}^{1}L^\infty}\\
&\lesssim& \Vert\Delta_{j}\omega^1_0\Vert_{B^{1}_{\infty,1}}+\Vert\frac{\widetilde{\omega}^{1}_{j}}{x_2}\Vert_{L_{t}^{\infty} B^{0}_{\infty,1}}+2^{2j}\Vert\Delta_{j}\rho(t)\Vert_{L_{t}^{1}L^\infty}\\
&+& \int_0^t e^{-CV(\tau)}\Vert\widetilde{\omega}_{j}^1(\tau)\Vert_{B^{1}_{\infty,1}}\Vert v(\tau)\Vert_{B^{1}_{\infty,1}}d\tau.
\end{eqnarray*}
According to Gronwall's inequality one obtain,
\begin{equation*}
e^{-CV(t)}\Vert\widetilde{\omega}_{j}^{1}(t)\Vert_{B^{1}_{\infty,1}}\lesssim \big(\Vert\Delta_{j}\omega^1_0\Vert_{B^{1}_{\infty,1}}+\Vert\frac{\widetilde{\omega}^{1}_{j}}{x_2}\Vert_{L_{t}^{\infty} B^{0}_{\infty,1}}+2^{2j}\Vert\Delta_{j}\rho(t)\Vert_{L_{t}^{1}L^{\infty}}\big)e^{CV(t)}.
\end{equation*}
Then 
\begin{equation}\label{z9}
\Vert\widetilde{\omega}_{j}^{1}(t)\Vert_{B^{1}_{\infty,1}}\lesssim \big(\Vert\Delta_{j}\omega^{1}_{0}\Vert_{B^{1}_{\infty,1}}+\Vert\frac{\widetilde{\omega}^{1}_{j}}{x_2}\Vert_{L_{t}^{\infty} B^{0}_{\infty,1}}+2^{2j}\Vert\Delta_{j}\rho(t)\Vert_{L_{t}^{1}L^{\infty}}\big)e^{CV(t)}.
\end{equation}
It remains to estimate $\Vert\frac{\widetilde{\omega}^{1}_{j}}{x_2}\Vert_{L_{t}^{\infty} B^{0}_{\infty,1}}.$ For this purpose we observe that $\frac{\widetilde{\omega}^{1}_{j}}{x_2}$ solves the equation,
\begin{equation*}
\left\{
\begin{array}{ll} 
\partial_{t}\frac{\widetilde{\omega}_{j}^{1}}{x_2}+v\cdot\nabla\frac{\widetilde{\omega}_{j}^{1}}{x_2}=\frac{\partial_{2}\Delta_{j}\rho}{x_2}\\ 
\frac{\widetilde{\omega}_{j}^{1}}{x_2}(t=0)=\frac{\Delta_{j}\omega_{0}^{1}}{x_2}
\end{array} \right.
\end{equation*} 
Applying again Proposition \ref{prop2} and H\"older inequality yields,
\begin{eqnarray}\label{z10}
\nonumber \Vert\frac{\widetilde{\omega}^{1}_{j}}{x_2}\Vert_{B^{0}_{\infty,1}}&\le& C e^{CV_{1}(t)}\bigg(\Vert\frac{\Delta_{j}\omega^{1}_{0}}{x_2}\Vert_{B^{0}_{\infty,1}}+\int_0^t e^{-CV_{1}(\tau)}\Vert\frac{\partial_{2}\Delta_{j}\rho(\tau)}{x_2}\Vert_{B^{0}_{\infty,1}}d\tau\bigg)\\
&\le& C e^{CV(t)}\Big(\Vert\frac{\Delta_{j}\omega^{1}_{0}}{x_2}\Vert_{B^{0}_{\infty,1}}+\Vert\frac{\partial_{2}\Delta_{j}\rho(t)}{x_2}\Vert_{L_{t}^{1}B^{0}_{\infty,1}}\Big).
\end{eqnarray}
By Taylor formula,
\begin{equation}\label{z11}
\partial_{2}\Delta_{j}\rho(x_1,x_2,z)=\partial_{2}\Delta_{j}\rho(x_1,0,z)+x_2 \int_0^1 \partial_{22}\Delta_{j}\rho(x_1,\tau x_2,z)d\tau.
\end{equation}
Since $\Delta_{j}\rho$ is an axisymmetric function and
\begin{eqnarray*}
\partial_{2}\Delta_{j}\rho(x_{1},x_{2},z)&=&\partial_{r}\Delta_{j}\rho(x_{1},x_{2},z)\partial_{2} r\\
&=&\partial_{r}\Delta_{j}\rho(x_{1},x_{2},z)\frac{x_2}{r},
\end{eqnarray*}
then $\partial_{2}\Delta_{j}\rho(x_1,x_2=0,z) =0.$ Thus
\begin{equation*}
\Vert\frac{\partial_{2}\Delta_{j}\rho}{x_2}\Vert_{B^{0}_{\infty,1}}\le \int_0^1 \Vert(\partial_{22}\Delta_{j}\rho)(\cdot,\tau\cdot,\cdot)\Vert_{B^{0}_{\infty,1}}d\tau.
\end{equation*}
At this stage we need to the following proposition (see \cite{ahs08} for the proof).
\begin{prop}\label{prop a5}
Let $h:\RR^3\longrightarrow \RR$ be a function such that $h \in B_{\infty,1}^0$ and we denote by $h_{\tau}(x_1,x_2,x_3)=h(x_1,\tau x_2,x_3)$. Then for every $0<\tau<1$, we have
$$\Vert h_{\tau}\Vert_{B_{\infty,1}^0}\le C(1-\log \tau)\Vert h \Vert_{B_{\infty,1}^0},$$
where $C$ is a absolute positive constant.
\end{prop}
Hence it follows that
\begin{eqnarray}\label{z12}
\nonumber \Vert\frac{\partial_{2}\Delta_{j}\rho(t)}{x_2}\Vert_{B^{0}_{\infty,1}}&\lesssim& \Vert\partial_{22}\Delta_{j}\rho(t)\Vert_{B_{\infty,1}^0}\int_0^1(1-\log \tau)d\tau\\
 &\lesssim& 2^{2j}\Vert\Delta_{j}\rho(t)\Vert_{L^{\infty}}.
\end{eqnarray}
Now to estimate $\Vert\frac{\Delta_{j}\omega^{1}_{0}}{x_2}\Vert_{B^{0}_{\infty,1}}$, we have $v^0$ is axisymmetric then by Proposition \ref{prop4}, $\Delta_j v^0$ is also axisymmetric. Consequently $\Delta_j \omega^0$ is the vorticity of an axisymmetric vector field hence $\Delta_j \omega^1_0(x_1,0,z)=0.$ Applying again Taylor formula we get  
$$\Delta_{j}\omega^{1}_{0}(x_1,x_2,z)=x_{2}\int_{0}^{1}\partial_{x_{2}} \Delta_{j}\omega^{1}_{0}(x_1,\tau x_2,z) d\tau.$$
Using Proposition \ref{prop a5} as above, we get easily
\begin{eqnarray}\label{z13}
\nonumber\Vert\frac{\Delta_{j}\omega_0^1}{x_2}\Vert_{B^{0}_{\infty,1}}&\le& \int_{0}^{1}\Vert(\partial_{x_{2}} \Delta_{j}\omega^{1}_{0})(\cdot,\tau\cdot,\cdot)\Vert_{B^{0}_{\infty,1}}d\tau\\  
\nonumber &\lesssim& \Vert(\partial_{x_{2}} \Delta_{j}\omega^{1}_{0})\Vert_{B^{0}_{\infty,1}}\int_{0}^{1}(1-\log \tau)d\tau\\
&\lesssim& 2^{j}\Vert\Delta_{j}\omega_0\Vert_{L^\infty}.
\end{eqnarray}
Plugging \eqref{z12} and \eqref{z13} into \eqref{z10}, we find
\begin{eqnarray}\label{z14}
\nonumber \Vert\frac{\widetilde{\omega}^{1}_{j}}{x_2}\Vert_{B^{0}_{\infty,1}}&\lesssim& \Big(2^{j}\Vert\Delta_{j}\omega_0\Vert_{L^\infty}+\Vert\frac{\partial_{2}\Delta_{j}\rho(t)}{x_2}\Vert_{L_{t}^{1} B^{0}_{\infty,1}}\Big)e^{CV(t)}\\
&\lesssim& \Big(2^{j}\Vert\Delta_{j}\omega_0\Vert_{L^\infty}+2^{2j}\Vert\Delta_{j}\rho(t)\Vert_{L_{t}^{1} L^{\infty}}\Big)e^{CV(t)}.
\end{eqnarray}
Plugging now the estimate \eqref{z14} into \eqref{z9} and using the embedding $L_{t}^{1}B_{\infty,1}^{1}\hookrightarrow L_{t}^{1}Lip,$ we obtain
\begin{equation*}
\Vert\widetilde{\omega}^{1}_{j}(t)\Vert_{B^{1}_{\infty,1}}\lesssim \big(2^{j}\Vert\Delta_{j}\omega_0\Vert_{L^\infty}+2^{2j}\Vert\Delta_{j}\rho(t)\Vert_{L_{t}^{1} L^{\infty}}\big)e^{CV(t)}.
\end{equation*}
This can be written as 
\begin{equation*}
\Vert\Delta_{k}\widetilde{\omega}^{1}_{j}(t)\Vert_{L^{\infty}}\lesssim 2^{j-k} e^{CV(t)}\big(\Vert\Delta_{j}\omega_0\Vert_{L^\infty}+2^{j}\Vert\Delta_{j}\rho(t)\Vert_{L_{t}^{1} L^{\infty}}\big).
\end{equation*}
Similar arguments gives the same estimate for $\widetilde{\omega}^{2}_{j}.$ Finally we obtain,
\begin{equation*}
\Vert\Delta_{k}\widetilde{\omega}_{j}(t)\Vert_{L^{\infty}}\lesssim 2^{j-k} e^{CV(t)}\big(\Vert\Delta_{j}\omega_0\Vert_{L^\infty}+2^{j}\Vert\Delta_{j}\rho(t)\Vert_{L_{t}^{1} L^{\infty}}\big).
\end{equation*}
The proof of the Proposition is now achieved.
\end{proof}
We will now see how use this to estimate $\Vert\omega(t)\Vert_{B^{0}_{\infty,1}}$ and then to obtain a Lipschitz norm of the velocity $v.$
\begin{prop}\label{prop a6} Let $v^{0}\in L^{2}$ such that $\frac{\omega^0}{r}\in L^{3,1}$, $\rho^0\in L^{2}\cap L^{p}$ with $p>6$ and such that $\vert x_{h}\vert^{2}\rho^{0}\in L^{2}$ and if in addition  $\omega^{0}\in B^{0}_{\infty,1}$. Then the system \eqref{a} satisfies for every $t \in \RR_+,$ 
$$\Vert\omega(t)\Vert_{B^{0}_{\infty,1}}+\Vert\nabla v(t)\Vert_{L^{\infty}}\lesssim \Phi_{5}(t).$$
\end{prop}
\begin{proof}
Let $M$ be a fixed positive integer that will be chosen later. Then we have from the definition of Besov spaces and Proposition \ref{prop a4}-(1),
\begin{eqnarray}\label{s1}
\nonumber \Vert\omega(t)\Vert_{B^{0}_{\infty,1}}&\le& \sum_{k}\Vert\Delta_{k}\sum_{j}\widetilde{\omega}_{j}(t)\Vert_{L^{\infty}}\\
\nonumber &\lesssim& \sum_{k}\bigg(\sum_{\vert k-j \vert \ge M}\Vert\Delta_{k}\widetilde{\omega}_{j}(t)\Vert_{L^{\infty}}\bigg)+\sum_{k}\bigg(\sum_{\vert k-j \vert < M}\Vert\Delta_{k}\widetilde{\omega}_{j}(t)\Vert_{L^{\infty}}\bigg)\\ 
&:=& \textnormal{I}_{1}+\textnormal{I}_{2}.
\end{eqnarray}
To estimate the term $\textnormal{I}_{1}$ we use Proposition \ref{prop a4}-(4), the convolution inequality for the series and the inequality \eqref{dh},
\begin{eqnarray}\label{s2}
\nonumber \textnormal{I}_{1}&=& \sum_{k}\bigg(\sum_{\vert k-j \vert \ge M}\Vert\Delta_{k}\widetilde{\omega}_{j}(t)\Vert_{L^{\infty}}\bigg)\\ 
\nonumber &\lesssim& 2^{-M}(\Vert\omega^{0}\Vert_{B^{0}_{\infty,1}}+\Vert\rho(t)\Vert_{L_{t}^{1}B^{1}_{\infty,1}})e^{C V(t)}\\
&\lesssim& 2^{-M}(\Vert\omega^{0}\Vert_{B^{0}_{\infty,1}}+\Phi_{4}(t))e^{C V(t)}
\end{eqnarray}
For the second term $\textnormal{I}_{2}$ we use the continuity of the operator $\Delta_k$ in the space $L^\infty$, Proposition \ref{prop a4}-(3) and \eqref{dh} and obtain
\begin{eqnarray}\label{s3}
\nonumber \textnormal{I}_{2}&\lesssim& \sum_{j}\sum_{\vert k-j \vert < M}\Vert\widetilde{\omega}_{j}(t)\Vert_{L^{\infty}}\\
\nonumber &\lesssim& \Phi_{3}(t)(\Vert\omega^{0}\Vert_{B^{0}_{\infty,1}}+\Vert\rho(t)\Vert_{L_{t}^{1}B^{1}_{\infty,1}})\\
\nonumber &\lesssim& \Phi_{3}(t) M (\Vert\omega^{0}\Vert_{B^{0}_{\infty,1}}+\Phi_{4}(t))\\
&\lesssim& \Phi_{4}(t) M. 
\end{eqnarray}
Combining now \eqref{s1}, \eqref{s2} and \eqref{s3} we find
\begin{eqnarray*}
\Vert\omega(t)\Vert_{B^{0}_{\infty,1}}\lesssim (2^{-M}e^{C V(t)}+\Phi_{4}(t)M).
\end{eqnarray*}
We choose $M$ such that
$$M=[C V(t)]+1,$$ then
\begin{equation}\label{s4}
\Vert\omega(t)\Vert_{B^{0}_{\infty,1}}\lesssim (V(t)+1) \Phi_{4}(t).
\end{equation}
It remains to estimate $V(t)$, then for this purpose we have by definition of Besov space, Bernstein inequality, Proposition \ref{prop a1} and the estimate $\Vert\Delta_{j}v \Vert_{L^\infty}\sim 2^{-j}\Vert\Delta_{j}\omega \Vert_{L^\infty},$
\begin{eqnarray*}
\Vert v(t)\Vert_{B^{1}_{\infty,1}}&=& \sum_{j \ge -1} 2^{j}\Vert\Delta_{j}v(t)\Vert_{L^\infty}\\
&\lesssim& \Vert\Delta_{-1}v(t)\Vert_{L^2}+\sum_{j \ge 0}\Vert\Delta_{j}\omega(t)\Vert_{L^\infty}\\
&\lesssim& \Vert v(t)\Vert_{L^2}+\Vert\omega(t)\Vert_{B^{0}_{\infty,1}}\\
&\lesssim& C_{0}(1+t)+\Phi_{4}(t)\Big(1+\int_{0}^{t}\Vert v(\tau)\Vert_{B^{1}_{\infty,1}}d\tau\Big). 
\end{eqnarray*}
We have used above \eqref{s4}, thus by Gronwall's inequality we obtain
\begin{equation}\label{s5}
\Vert v(t)\Vert_{B^{1}_{\infty,1}}\lesssim \Phi_{5}(t).
\end{equation}
Plugging this estimate in \eqref{s4} gives,
$$\Vert\omega(t)\Vert_{B^{0}_{\infty,1}}\lesssim \Phi_{5}(t).$$
Using now the embeddings $B^{1}_{\infty,1}\hookrightarrow Lip(\RR^3)$ and \eqref{s5} we get
$$\Vert\nabla v(t)\Vert_{L^\infty}\lesssim \Phi_{5}(t).$$  
\end{proof}
\subsubsection{Strong a priori estimates} The task is now to find some global estimates for stronger norms of the solution of \eqref{a}. 
\begin{prop}\label{prop a7}
Let  $v^0 \in B^{\frac{5}{2}}_{2,1}$ be a divergence free axisymmetric vector field without swirl and $\rho^0 \in B^{\frac{1}{2}}_{2,1}\cap L^p$ with $p>6$ an axisymmetric function such that $\vert x_{h}\vert^{2}\rho^{0}\in L^2.$ Then any smooth solution $(v,\rho)$ of the system \eqref{a} satisfies
\begin{equation*}
\Vert v \Vert_{\widetilde{L}_{t}^{\infty}B^{\frac{5}{2}}_{2,1}}+\Vert\rho\Vert_{\widetilde{L}_{t}^{\infty}B^{\frac{1}{2}}_{2,1}}+\Vert\rho\Vert_{\widetilde{L}_{t}^{1}B^{\frac{5}{2}}_{2,1}}\lesssim \Phi_{6}(t).
\end{equation*}
\end{prop}
Recall that for every $T>0,$ $\rho\ge 1,$ $(p,r)\in[1,\infty]^{2}$ and $s\in\RR$ the Chemin-Lerner space $\widetilde{L}_{T}^{\rho}B_{p,r}^{s}$ is defined as the set of all distribution $f$ satisfying
$$\Vert f\Vert_{\widetilde{L}_{T}^{\rho}B_{p,r}^{s}}:=\Big\Vert (2^{qs}\Vert\Delta_{q}f\Vert_{L_{T}^{\rho}L^p})_{q}\Big\Vert_{\ell^r}.$$
\begin{proof}
We localize in frequency the equation of the velocity, then we have for every $j \ge -1,$
$$\partial_{t}\Delta_{j} v+v\cdot\nabla\Delta_{j} v+\nabla\Delta_{j}p=\Delta_{j}\rho e_{z}-[\Delta_{j},v\cdot\nabla]v.$$ 
Taking the $L^2$- scalar product with $\Delta_{j}v$ and using H\"older inequality,
$$\frac{1}{2}\frac{d}{dt}\Vert\Delta_{j} v(t)\Vert^{2}_{L^2}\le \Vert\Delta_{j} v(t)\Vert_{L^2}\Vert\Delta_{j} \rho(t)\Vert_{L^2}+\Vert\Delta_{j} v(t)\Vert_{L^2}\Vert[\Delta_{j},v\cdot\nabla]v(t)\Vert_{L^2}.$$
This implies that
$$\frac{d}{dt}\Vert\Delta_{j} v(t)\Vert_{L^2}\le \Vert\Delta_{j}\rho(t)\Vert_{L^2}+ \Vert[\Delta_{j},v\cdot\nabla]v(t)\Vert_{L^2}.$$
Integrating in time we obtain
$$\Vert\Delta_{j} v(t)\Vert_{L^2}\le \Vert\Delta_{j} v^{0}\Vert_{L^2}+\Vert\Delta_{j}\rho(t)\Vert_{L_{t}^{1} L^2}+ \Vert[\Delta_{j},v\cdot\nabla]v(t)\Vert_{L_{t}^{1}L^2}.$$
Multiplying the above inequality by $2^{\frac{5}{2}j}$ and taking the $\ell^{1}$ -norm we obtain thus
\begin{equation}\label{s6}
\Vert v(t)\Vert_{B^{\frac{5}{2}}_{2,1}}\le \Vert v^{0}\Vert_{B^{\frac{5}{2}}_{2,1}}+\Vert\rho(t)\Vert_{\widetilde{L}_{t}^{1}B^{\frac{5}{2}}_{2,1}}+\Vert(2^{\frac{5}{2}j}\Vert[\Delta_{j},v\cdot\nabla]v(t)\Vert_{L_{t}^{1}L^{2}})_{j}\Vert_{\ell^{1}}.
\end{equation}
To estimate the commutator term, we use the following Lemma (see \cite{che98} for the proof).
\begin{lem} \label{m1}
Let $\eta$ be a smooth function and $v$ be a smooth vector field of $\RR^{3}$ with zero divergence. Then we have for all $1 \le p \le \infty$  and $s \ge -1,$
\begin{equation*}
\sum_{j \ge -1}2^{js}\Vert[\Delta_{j}, v\cdot\nabla]\eta\Vert_{L^p}\lesssim\left\{ 
\begin{array}{ll} 
\Vert\nabla v\Vert_{L^\infty}\Vert\eta\Vert_{B_{p,1}^s}\quad\hbox{if}\quad -1<s< 1\\ 
\Vert\nabla v\Vert_{L^\infty}\Vert\eta\Vert_{B_{p,1}^s}+\Vert\nabla\eta\Vert_{L^\infty}\Vert v \Vert_{B_{p,1}^s} \quad\hbox{if}\quad 1\le s
\end{array} \right.
\end{equation*}
\end{lem}
Therefore we obtain in \eqref{s6},
\begin{equation*}
\Vert v(t)\Vert_{B^{\frac{5}{2}}_{2,1}}\le \Vert v^{0}\Vert_{B^{\frac{5}{2}}_{2,1}}+\Vert\rho(t)\Vert_{\widetilde{L}_{t}^{1}B^{\frac{5}{2}}_{2,1}}+C\int_{0}^{t}\Vert\nabla v(\tau) \Vert_{L^\infty}\Vert v(\tau)\Vert_{B^{\frac{5}{2}}_{2,1}}d\tau.
\end{equation*}
Using Gronwall's inequality we get
\begin{eqnarray}\label{s7}
\nonumber \Vert v(t)\Vert_{B^{\frac{5}{2}}_{2,1}}&\lesssim& (\Vert v^{0}\Vert_{B^{\frac{5}{2}}_{2,1}}+\Vert\rho(t)\Vert_{\widetilde{L}_{t}^{1}B^{\frac{5}{2}}_{2,1}})e^{C \int_{0}^{t}\Vert\nabla v(\tau) \Vert_{L^\infty}d\tau}\\
&\lesssim& (\Vert v^{0}\Vert_{B^{\frac{5}{2}}_{2,1}}+\Vert\rho(t)\Vert_{\widetilde{L}_{t}^{1}B^{\frac{5}{2}}_{2,1}}) \Phi_{6}(t),
\end{eqnarray}
where we have used Proposition \ref{prop a6}. It remains then to estimate $\Vert\rho(t)\Vert_{\widetilde{L}_{t}^{1}B^{\frac{5}{2}}_{2,1}}.$ For this purpose we localize in  frequency the equation of the density. For $j\ge 0$ we have,
$$\partial_{t}\Delta_{j}\rho+v\cdot\nabla\Delta_{j}\rho-\Delta\Delta_{j}\rho=-[\Delta_{j},v\cdot\nabla]\rho.$$    
Taking again the $L^2$- scalar product with $\Delta_{j}\rho$ and using H\"older's inequality,
$$\frac{1}{2}\frac{d}{dt}\Vert\Delta_{j}\rho(t)\Vert^{2}_{L^2}-\int_{\RR^3}(\Delta\Delta_{j}\rho)\Delta_{j}\rho dx \le 
\Vert\Delta_{j}\rho(t)\Vert_{L^2}\Vert[\Delta_{j},v\cdot\nabla]\rho(t)\Vert_{L^2}.$$
Using now the generalized Bernstein inequality see \cite{d01,gl02}
$$\frac{1}{2}2^{2j}\Vert\Delta_{j}\rho(t)\Vert^{2}_{L^2}\le -\int_{\RR^3}(\Delta\Delta_{j}\rho)\Delta_{j}\rho dx.$$
Hence,
$$\frac{d}{dt}\Vert\Delta_{j}\rho(t)\Vert_{L^2}+c 2^{2j}\Vert\Delta_{j}\rho(t)\Vert_{L^2}\lesssim \Vert[\Delta_{j},v\cdot\nabla]\rho(t)\Vert_{L^2}.$$
This gives,
$$\frac{d}{dt}(e^{ct 2^{2j}}\Vert\Delta_{j}\rho(t)\Vert_{L^2})\le e^{ct 2^{2j}} \Vert[\Delta_{j},v\cdot\nabla]\rho(t)\Vert_{L^2}.$$
It follows that,
\begin{equation}\label{s8}
\Vert\Delta_{j}\rho(t)\Vert_{L^2}\lesssim e^{-ct 2^{2j}}\Vert\Delta_{j}\rho^{0}\Vert_{L^2}+\int_{0}^{t}e^{-c(t-\tau)2^{2j}}
\Vert[\Delta_{j},v\cdot\nabla]\rho(\tau)\Vert_{L^2}d\tau
\end{equation}
Integrating in time implies that
\begin{equation}\label{s9}
\Vert\Delta_{j}\rho(t)\Vert_{L_{t}^{1}L^2}\lesssim 2^{-2j}\Vert\Delta_{j}\rho^{0}\Vert_{L^2}+2^{-2j}
\Vert[\Delta_{j},v\cdot\nabla]\rho(t)\Vert_{L_{t}^{1}L^2}.
\end{equation}
From \eqref{s8} and \eqref{s9} we obtain for $j \ge 0$
$$\Vert\Delta_{j}\rho(t)\Vert_{L_{t}^{\infty}L^2}+2^{2j}\Vert\Delta_{j}\rho(t)\Vert_{L_{t}^{1}L^2}\lesssim \Vert\Delta_{j}\rho^{0}\Vert_{L^2}+\Vert[\Delta_{j},v\cdot\nabla]\rho(t)\Vert_{L_{t}^{1}L^2}.$$
Multiplying the above inequality by $2^{\frac{j}{2}}$ and taking the $\ell^1$ norm we find,
\begin{eqnarray*}
\Vert\rho(t)\Vert_{\widetilde{L}_{t}^{\infty}B^{\frac{1}{2}}_{2,1}}+\Vert\rho(t)\Vert_{\widetilde{L}_{t}^{1}B^{\frac{5}{2}}_{2,1}}&\lesssim& \Vert\Delta_{-1}\rho(t)\Vert_{L_{t}^{1}L^2}+\Vert\rho^{0}\Vert_{B^{\frac{1}{2}}_{2,1}}+\Vert(2^{\frac{j}{2}}\Vert[\Delta_{j},v\cdot\nabla]\rho(t)\Vert_{L_{t}^{1}L^{2}})_{j}\Vert_{\ell^{1}}\\
&\lesssim& t\Vert\rho^{0}\Vert_{L^2}+\Vert\rho^{0}\Vert_{B^{\frac{1}{2}}_{2,1}}+\int_{0}^{t}\Vert\nabla v(\tau) \Vert_{L^\infty}\Vert\rho(\tau)\Vert_{B^{\frac{1}{2}}_{2,1}}d\tau\\
&\lesssim& \Vert\rho^{0}\Vert_{B^{\frac{1}{2}}_{2,1}}(1+t)+\int_{0}^{t}\Vert\nabla v(\tau) \Vert_{L^\infty}\Vert\rho(\tau)\Vert_{\widetilde{L}_{\tau}^{\infty}B^{\frac{1}{2}}_{2,1}}d\tau\\
&\lesssim& C_{0}(1+t)+\int_{0}^{t}\Phi_{5}(\tau)(\Vert\rho(\tau)\Vert_{\widetilde{L}_{\tau}^{\infty}B^{\frac{1}{2}}_{2,1}}+\Vert\rho(\tau)\Vert_{\widetilde{L}_{\tau}^{1}B^{\frac{5}{2}}_{2,1}})d\tau,
\end{eqnarray*}
where we have used Lemma \ref{m1} and Proposition \ref{prop a6}.\\
Finally using Gronwall's inequality we find
\begin{equation}\label{s10}
\Vert\rho(t)\Vert_{\widetilde{L}_{t}^{\infty}B^{\frac{1}{2}}_{2,1}}+\Vert\rho(t)\Vert_{\widetilde{L}_{t}^{1}B^{\frac{5}{2}}_{2,1}}\lesssim \Phi_{6}(t).
\end{equation}
Putting now \eqref{s10} into \eqref{s7} we find
$$\Vert v(t)\Vert_{B^{\frac{5}{2}}_{2,1}}\lesssim \Phi_{6}(t).$$
The proof of the Proposition is now complete.
\end{proof}
\subsection{Uniqueness result} We will prove the uniqueness result for the system \eqref{a} in the following space
$$\mathcal{A}_{T}:=(L_T^{\infty}L^2\cap L_T^{1}\dot{W}^{1,\infty})\times (L_T^{\infty}L^2\cap L_T^{1}\dot{W}^{1,\infty}).$$
We take two solutions $(v_j,\rho_j)$, with $j=1,2$ for $\textnormal{(1.1)}$ belonging to the space $\mathcal{A}_{T}$ for a fixed time $T>0$, with initial data $(v^0_j,\rho^0_j)$, $j=1,2$ and we denote
$$v=v_2-v_1\qquad\textnormal{and}\qquad \rho=\rho_2-\rho_1.$$
Then we find the equations
\begin{equation}\label{s11} 
\left\{ \begin{array}{ll} 
\partial_{t} v+v_{2}\cdot\nabla v+\nabla p=- v\cdot\nabla v_{1}+\rho e_{z}\\ 
\partial_{t}\rho+v_{2}\cdot\nabla\rho-\Delta\rho=-v\cdot\nabla\rho_{1}\\
v_{| t=0}= v^{0}, \quad \rho_{| t=0}=\rho^{0}.  
\end{array} \right.
\end{equation}  
Taking the $L^2$ inner product of the first equation of \eqref{s11} with $v$, integrating by parts and using H\"older inequality, we get
$$\frac{1}{2}\frac{d}{dt}\Vert v(t)\Vert^{2}_{L^2}\le \Vert v(t)\Vert^{2}_{L^2}\Vert\nabla v_{1}(t)\Vert_{L^\infty}+\Vert v(t)\Vert_{L^2}\Vert\rho(t)\Vert_{L^2}.$$
Integrating in time we get
\begin{equation*}
\Vert v(t)\Vert^{2}_{L^2}\le\Vert v^{0}\Vert^{2}_{L^2}+2\int_{0}^{t}\Vert v(\tau)\Vert^{2}_{L^2}\Vert\nabla v_{1}(\tau)\Vert_{L^\infty}d\tau+2\int_{0}^{t}\Vert\rho(\tau)\Vert_{L^2}\Vert v(\tau)\Vert_{L^2}d\tau.
\end{equation*}
Then we get
\begin{equation}\label{10}
\Vert v(t)\Vert^{2}_{L^2}\le\Vert v^{0}\Vert^{2}_{L^2}+2\int_{0}^{t}\Vert v(\tau)\Vert^{2}_{L^2}\Vert\nabla v_{1}(\tau)\Vert_{L^\infty}d\tau+2\int_{0}^{t}(\Vert\rho(\tau)\Vert^{2}_{L^2}+\Vert v(\tau)\Vert^{2}_{L^2})d\tau.
\end{equation}
As above, taking again the $L^2$ inner product of the second equation of \eqref{s11} with $\rho$ and integrating by parts, we get finally
\begin{eqnarray}\label{11}
\nonumber \Vert\rho(t)\Vert^{2}_{L^2}&\le& \Vert\rho^{0}\Vert^{2}_{L^2}+2\int_{0}^{t}\Vert v(\tau)\Vert_{L^2}\Vert \rho(\tau)\Vert_{L^2}\Vert\nabla\rho_{1}(\tau)\Vert_{L^\infty}d\tau\\
&\le& \Vert\rho^{0}\Vert^{2}_{L^2}+2\int_{0}^{t}\Big(\Vert v(\tau)\Vert^{2}_{L^2}+2\Vert \rho(\tau)\Vert^{2}_{L^2}\Big)\Vert\nabla\rho_{1}(\tau)\Vert_{L^\infty}d\tau
\end{eqnarray}
Putting $g(t)=\Vert v(t)\Vert^{2}_{L^2}+\Vert\rho(t)\Vert^{2}_{L^2},$ we obtain from \eqref{10} and \eqref{11} that,
$$g(t)\le g^{0}+2\int_{0}^{t}\Big(\Vert\nabla v_{1}(t)\Vert_{L^\infty}+\Vert\nabla\rho_{1}(t)\Vert_{L^\infty}+1\Big)g(\tau)d\tau.$$
Gronwall inequality yields
\begin{equation}\label{8}
g(t)\le g^{0} \exp\Big(2({\Vert\nabla v_{1}(t)\Vert_{L_{t}^{1}L^\infty}+\Vert\nabla\rho_{1}(t)\Vert_{L_{t}^{1}L^\infty}+t})\Big).
\end{equation}
This proves the uniqueness result.
\subsection{Existence} We will now construct a global solution for the Boussinesq system \eqref{a}. First we smooth out initial data
$$v^{n}_{0}=S_{n} v^0\;\;\textnormal{and}\;\;\rho^{n}_{0}=S_{n}\rho^0.$$ 
By definition of the operator $S_n$ : there is a positive radial function $\chi\in\mathcal{D}(\RR^3)$ such that
$$v^{n}_{0}:=S_{n} v^0=2^{3n}\chi(2^{n}\cdot)\ast v^0$$
and
$$\rho^{n}_{0}:=S_{n}\rho^0=2^{3n}\chi(2^{n}\cdot)\ast \rho^0,$$
We can easily prove the following result.
\begin{lem}\label{lem 58}
Let $v^{0}\in B_{2,1}^{\frac{5}{2}}$ be an axisymmetric vector field with zero divergence and such that $\frac{\omega_{0}}{r}\in L^{3,1}$ and let $\rho^{0}\in B^{\frac{1}{2}}_{2,1}\cap L^{p}$ with $p>6$ be an axisymmetric function such that $\vert x_{h}\vert^{2}\rho^{0}\in L^2.$ Then for every $n\in\NN^{\ast},$ the functions $v^{n}_{0}$ and $\rho^{n}_{0}$ are axisymmetric  and $\textnormal{div}\,v^{n}_{0}=0.$ Moreover there exist a constant $C$ such that,
\begin{eqnarray*}
\Vert v^{n}_{0}\Vert_{B^{\frac{5}{2}}_{2,1}}\le C\Vert v^{0}\Vert_{B^{\frac{5}{2}}_{2,1}},\;\;\;\Vert\frac{\omega^{n}_{0}}{r}\Vert_{L^{3,1}}\le C\Vert\frac{\omega_{0}}{r}\Vert_{L^{3,1}},\;\;\;\Vert\rho^{n}_{0}\Vert_{B^{\frac{1}{2}}_{2,1}\cap L^{p}}\le C\Vert\rho^{0}\Vert_{B^{\frac{1}{2}}_{2,1}\cap L^{p}}\\\textnormal{and}\;\;\;\; \sup_{n\in\NN}\Vert\vert x_{h}\vert^{2}\rho_{0}^{n}\Vert_{L^2}\le C(\Vert\rho^{0}\Vert_{L^2}+\Vert\vert x_{h}\vert^{2}\rho^{0}\Vert_{L^2}).
\end{eqnarray*}
\end{lem}
\begin{proof}
The fact that the functions $v^{n}_{0}$ and $\rho^{n}_{0}$ are axisymmetric is due to the radial property of the function $\chi$ and the fact that $\textnormal{div}\,v^{n}_{0}=0$ due to the condition incompressible of the vector field $v^0.$
Now we have the cut-off in frequency is uniformly bounded in Lesbesque and Sobolev spaces, that is by applying convolution inequality we immediate get
\begin{eqnarray*}
\Vert\rho^{n}_{0}\Vert_{L^{p}}&\le& \Vert 2^{3n}\chi(2^{n}\cdot)\Vert_{L^{1}}\Vert\rho^{0}\Vert_{L^p}\\
&\le& \Vert\chi\Vert_{L^1}\Vert\rho^{0}\Vert_{L^p}\\
&\le& C \Vert\rho^{0}\Vert_{L^p}.
\end{eqnarray*} 
and in the Besov space, we have
\begin{eqnarray*}
\Vert v^{n}_{0}\Vert_{B^{\frac{5}{2}}_{2,1}}&\le& \sum_{p\le n-1}\Vert\Delta_{p}v^{0}\Vert_{B^{\frac{5}{2}}_{2,1}}\\
&\le& \sum_{\vert j-p \vert\le 1}2^{\frac{5}{2}j}\sum_{p\le n-1}\Vert\Delta_{j}\Delta_{p}v^{0}\Vert_{L^{2}}\\
&\le& C \sum_{j}2^{\frac{5}{2}j}\Vert\Delta_{j}v^{0}\Vert_{L^{2}}\\
&\le& C\Vert v^{0}\Vert_{B^{\frac{5}{2}}_{2,1}}.
\end{eqnarray*} 
Similarly for $\Vert\rho^{n}_{0}\Vert_{B^{\frac{1}{2}}_{2,1}},$ we get
$$\Vert\rho^{n}_{0}\Vert_{B^{\frac{1}{2}}_{2,1}}\le C \sum_{j}2^{\frac{1}{2}j}\Vert\Delta_{j}\rho^{0}\Vert_{L^2}\le C\Vert\rho^{0}\Vert_{B^{\frac{1}{2}}_{2,1}}.$$
For the estimate of $\Vert\frac{\omega^{n}_{0}}{r}\Vert_{L^{3,1}}$, we use the convolution inequality on Lorentz space $L^{3,1},$ we obtain as before
\begin{eqnarray*}
\Vert\frac{\omega^{n}_{0}}{r}\Vert_{L^{3,1}}=\Vert\frac{S_{n}\omega_{0}}{r}\Vert_{L^{3,1}}&\le& \Vert 2^{3n}\chi(2^{n}\cdot)\Vert_{L^{1}}\Vert\frac{\omega_{0}}{r}\Vert_{L^{3,1}}\\
&\le& C \Vert\frac{\omega_{0}}{r}\Vert_{L^{3,1}}.
\end{eqnarray*}
We prove the uniform boundedness of the moment of the density. First we write,
\begin{eqnarray*}
\vert x_{h}\vert^{2}\Big\vert\rho^{n}_{0}(x)\Big\vert &=& \vert x_{h}\vert^{2}\Big\vert 2^{3n}\int_{\RR^3}\chi(2^{n}(x-y))\rho^{0}(y)dy\Big\vert\\
&\le& 2 2^{3n}\int_{\RR^3}\vert x_{h}-y_{h}\vert^{2}\chi(2^{n}(x-y))\rho^{0}(y)dy+2 2^{3n}\Big\vert\int_{\RR^3}\chi(2^{n}(x-y))\vert y_{h}\vert^{2}\rho^{0}(y)dy\Big\vert\\
&\le& 2 2^{-2n} (2^{3n}\chi_{1}(2^{n}\cdot)\ast\rho^{0})(x)+2(2^{3n}\vert\chi\vert(2^{n}\cdot)\ast(\vert y_{h}\vert^{2}\rho^{0}))(x),
\end{eqnarray*}
where $\chi_{1}(x)=\vert x_{h}\vert^{2}\chi(x).$ Now from convolution inequality we obtain,
$$\Vert\vert x_{h}\vert^{2}\rho^{n}_{0}\Vert_{L^2}\le C 2^{-2n}\Vert\rho^{0}\Vert_{L^2}+C\Vert\vert x_{h}\vert^{2}\rho^{0}\Vert_{L^2}.$$
This gives that,
$$\sup_{n\in\NN} \Vert\vert x_{h}\vert^{2}\rho^{n}_{0}\Vert_{L^2}\le C(\Vert\rho^{0}\Vert_{L^2}+\Vert\vert x_{h}\vert^{2}\rho^{0}\Vert_{L^2}).$$
\end{proof}
Let us now consider the following system
\begin{equation}\label{s15} 
\left\{ \begin{array}{ll} 
\partial_{t} v^n+v^{n}\cdot\nabla v^n+\nabla p_n=\rho^n e_{z}\\ 
\partial_{t}\rho^n+v^{n}\cdot\nabla\rho^n-\Delta\rho^n=0\\
\textnormal{div}\,v^{n}=0\\
v^{n}_{| t=0}=S_{n} v^{0}, \quad \rho^{n}_{| t=0}=S_{n}\rho^{0}.  
\end{array} \right.
\end{equation}
Lemma \ref{lem 58} gives that the initial data are smooths and axisymmetrics. Thus we can construct locally in time a unique solution $(v^n,\rho^n).$ This solution is globally defined since the Lipschitz norm of the velocity does not blow up in finite time by Proposition \ref{prop a6}. Once again from the a priori estimates we have
\begin{equation}\label{ss15}
\Vert v^{n}\Vert_{\widetilde{L}_T^\infty B_{2,1}^{\frac 52}}+\Vert\rho^{n}\Vert_{\widetilde{L}_T^\infty B_{2,1}^{\frac 12}}+\Vert\rho^{n}\Vert_{\widetilde{L}_T^1 B_{2,1}^{\frac 52}}\lesssim\Phi_{6}(T).
\end{equation}
The control is uniform with respect to the parameter $n.$ Thus it follows that up to an extraction the sequence $(v^n,\rho^n)_{n\in\NN}$ is weakly convergent to some $(v,\rho)$ belonging to $\widetilde{L}_T^\infty B_{2,1}^{\frac 52}\times\widetilde{L}_T^\infty B_{2,1}^{\frac 12}\cap\widetilde{L}_T^1 B_{2,1}^{\frac 52}.$
Now we will prove that this sequence $(v^n,\rho^n)_{n\in\NN}$ converges strongly to $(v,\rho)$ in the space $(L_T^\infty L^2)^2.$
Let $\xi_{n,n^{\prime}}:=v^{n}-v^{n^{\prime}}$ and $\eta_{n,n^{\prime}}:=\rho^{n}-\rho^{n^{\prime}}$ then according to the estimate \eqref{8} and Propositions \ref{prop a3} and \ref{prop a6}, we get 
\begin{eqnarray*}
\Vert\xi_{n,n^{\prime}}\Vert_{L_T^\infty L^2}+\Vert\eta_{n,n^{\prime}}\Vert_{L_T^\infty L^2}\le (\Vert S_{n} v^{0}-S_{n^{\prime}} v^{0}\Vert_{L^2}+\Vert S_{n}\rho^{0}-S_{n^{\prime}}\rho^{0}\Vert_{L^2})\Phi_{6}(t).
\end{eqnarray*} 
This shows that the family $(v^n,\rho^{n})_{n\in\NN}$ is Cauchy sequence in the space $(L_T^\infty L^2)^2.$ Hence it converges strongly to $(v,\rho).$ Combining this result with \eqref{ss15} and by interpolation argument, we can get the strong convergences of $(v^n,\rho^n)$ to $(v,\rho)$ in $L_{T}^{\infty}H^{s}\times L_{T}^{\infty}H^{s^\prime}$ with $0\le s<\frac{5}{2}$ and $0\le s^{\prime}<\frac{1}{2}.$ The passage to the limit in the linear parts can be checked for example in the weak sense. For the nonlinear term $v^{n}\cdot\nabla v^{n}$, we write
$$v^{n}\cdot\nabla v^{n}=\textnormal{div}(v^{n}\otimes v^{n}).$$
Now, since $v^{n}\rightarrow v$ in $L_{T}^{\infty}H^{s}$and by using the fact that $H^{s}\cap L^\infty$ with $s>0$ is an algebra, we obtain $v^{n}\otimes v^{n}\rightarrow v\otimes v$ in $L_{T}^{\infty}H^{s}$ and thus $\textnormal{div}(v^{n}\otimes v^{n})\rightarrow \textnormal{div}(v\otimes v)$ in $L_{T}^{\infty}H^{s-1}.$\\
On the other-hand, $v^{n}\cdot\nabla\rho^{n}=\textnormal{div}(v^{n}\rho^{n}),$ and since $v^{n}\rightarrow v$ in $L_{T}^{\infty}L^{\infty}$ and $\rho^{n}\rightarrow\rho$ in $L_{T}^{\infty}L^{2}$ then we get $v^{n}\rho^{n}\rightarrow v\rho$ in $L_{T}^{\infty}L^{2}$ and consequently $\textnormal{div}(v^{n}\rho^{n})\rightarrow\textnormal{div}(v\rho)$ in $L_{T}^{\infty}H^{-1}.$ 
This allows us to pass to the limit in the system \eqref{s15} and we get that $(v,\rho)$ is a solution of the system \eqref{a}.\\\\
Let us now sketch the proof of the contiuity in time of the velocity. Let $\varepsilon>0$, $N\in\NN^{*}$ and $T>0$, then for every $t,t^{\prime} \in \RR_+,$
\begin{eqnarray*}
\Vert v(t)-v(t^{\prime})\Vert_{B^{\frac{5}{2}}_{2,1}}&\le& \sum_{j\le N}2^{\frac{5}{2}j}\Vert\Delta_{j}(v(t)-v(t^{\prime}))\Vert_{L^2}+2 \sum_{j>N}2^{\frac{5}{2}j}\Vert\Delta_{j}v\Vert_{L^2}\\
&\lesssim& 2^{\frac{5}{2}N}\Vert v(t)-v(t^\prime) \Vert_{L^{2}}+2 \sum_{j>N}2^{\frac{5}{2}j}\Vert\Delta_{j}v\Vert_{L_{T}^{\infty}L^{2}}.
\end{eqnarray*}
Since $v \in \widetilde{L}_{T}^{\infty}B^{\frac{5}{2}}_{2,1}$, then there exists $N$ sufficiently large such that
$$\sum_{j>N}2^{\frac{5}{2}j}\Vert\Delta_{j}v\Vert_{L_{T}^{\infty}L^{2}}\le \frac{\varepsilon}{4}$$
Therefore, we have
\begin{equation}\label{s16}
\Vert v(t)-v(t^{\prime})\Vert_{B^{\frac{5}{2}}_{2,1}}\lesssim 2^{\frac{5}{2}N}\Vert v(t)-v(t^\prime) \Vert_{L^{2}}+\frac{\varepsilon}{2}.
\end{equation}
It remains then to estimate $\Vert v(t)-v(t^\prime) \Vert_{L^{2}}.$ For this purpose we use the velocity equation, we have 
\begin{equation*}
\partial_{t}v=-\mathcal{P}(v\cdot\nabla v)+\mathcal{P}(\rho e_{z}).
\end{equation*}
Where $\mathcal{P}$ denote Leray projector. The solution of this equation is given by Duhamel formula,
$$v(t,x)=v^{0}(x)-\int^{t}_{0} \mathcal{P}(v\cdot\nabla v)(\tau) d\tau+\int^{t}_{0} \mathcal{P}\rho(\tau) d\tau.$$
Hence it follows that for $t,t^{\prime}\in \RR_+,$
$$v(t,x)-v(t^\prime, x)=-\int^{t}_{t^\prime} \mathcal{P}(v\cdot\nabla v)(\tau) d\tau+\int^{t}_{t^\prime} \mathcal{P}\rho(\tau) d\tau.$$
Taking the $L^2$ norm of the above equation and using the fact that the Leray projector $\mathcal{P}$ is continue into $L^2$, we get with $t^{\prime}\le t$ that 
\begin{eqnarray}\label{s17}
\nonumber \Vert v(t)-v(t^\prime) \Vert_{L^2}&\le& \int_{t^\prime}^{t}\Vert\mathcal{P}(v\cdot\nabla v)(\tau)\Vert_{L^2} d\tau+\int_{t^\prime}^{t}\Vert\mathcal{P}\rho(\tau)\Vert_{L^2}d\tau\\
\nonumber &\lesssim& \int_{t^\prime}^{t}\Vert v(\tau)\Vert_{L^2}\Vert\nabla v(\tau)\Vert_{L^\infty}d\tau+\int_{t^\prime}^{t}\Vert\rho(\tau)\Vert_{L^2}d\tau\\
\nonumber &\lesssim& \int_{t^\prime}^{t}\Vert v(\tau)\Vert_{L^2}\Vert\nabla v(\tau)\Vert_{L^\infty}d\tau+\int_{t^\prime}^{t}\Vert\rho^{0}\Vert_{L^{2}}d\tau\\
&\lesssim& \vert t-t^\prime \vert\big(\Vert v \Vert_{L_{t}^{\infty}L^2}\Vert\nabla v \Vert_{L_{t}^{\infty}L^\infty}+\Vert\rho^{0}\Vert_{L^2}\big).
\end{eqnarray}
We have used H\"older inequality and integration by parts for the first term of the above inequality and Proposition \ref{prop a1}-(a) for the second.\\
Putting now \eqref{s17} into \eqref{s16} and using Propositions \ref{prop a1} and \ref{prop a6}, we get
\begin{eqnarray*}
\Vert v(t)-v(t^{\prime})\Vert_{B^{\frac{5}{2}}_{2,1}}&\lesssim& 2^{\frac{5}{2}N}\vert t-t^{\prime} \vert \big(\Vert v \Vert_{L_{t}^{\infty}L^{2}}\Vert\nabla v \Vert_{L_{t}^{\infty} L^{\infty}}+\Vert\rho^{0}\Vert_{L^{2}}\big)+\frac{\varepsilon}{2}\\
&\lesssim& 2^{\frac{5}{2}N}\vert t-t^{\prime} \vert \big(\Phi_{5}(t)(1+t)+\Vert\rho^{0}\Vert_{L^{2}}\big)+\frac{\varepsilon}{2}. 
\end{eqnarray*}
It is enough then to choose $\vert t-t^{\prime}\vert<\beta$ such that
$$2^{\frac{5}{2}N}\vert t-t^{\prime}\vert\big((1+t)\Phi_{5}(t)+\Vert\rho^{0}\Vert_{L^{2}}\big)<\frac{\varepsilon}{2}.$$
Finally we obtain
$$\Vert v(t)-v(t^{\prime})\Vert_{B^{\frac{5}{2}}_{2,1}}\lesssim \varepsilon.$$
This proves the continuity in time of the velocity.\\
Let us now prove the continuity in time for the density. We first prove it in the Besov spaces $B^{\frac{1}{2}}_{2,1}.$ Similarly to the velocity, we write since $\rho \in \widetilde{L}^{\infty}_{T}B_{2,1}^{\frac{1}{2}}$ that for every $t,t^{\prime} \in \RR_+,$
\begin{eqnarray}\label{s18}
\nonumber \Vert \rho(t)-\rho(t^{\prime})\Vert_{B^{\frac{1}{2}}_{2,1}}&\le& \sum_{j\le N}2^{\frac{j}{2}}\Vert\Delta_{j}\rho(t)-\Delta_{j}\rho(t^{\prime})\Vert_{L^2}+2 \sum_{j>N}2^{\frac{j}{2}}\Vert\Delta_{j}\rho\Vert_{L_{T}^{\infty}L^2}\\
&\lesssim& \sum_{j\le N}2^{\frac{j}{2}}\Vert\Delta_{j}\rho(t)-\Delta_{j}\rho(t^\prime)\Vert_{L^{2}}+2 \frac{\varepsilon}{4}.
\end{eqnarray}
From the equation for $\rho$, we have 
$$\Delta_{j}\rho(t)=\Delta_{j}\rho^{0}(x)+\int_{0}^{t}\Delta\Delta_{j}\rho(\tau)d\tau-\int_{0}^{t}\Delta_{j}(v\cdot\nabla\rho)(\tau)d\tau$$
and similar for $\Delta_{j}\rho(t^{\prime}).$
Thus it follows from H\"older and Bernstein inequalities with $t^{\prime}\le t$ that, 
\begin{eqnarray}\label{s19}
\nonumber \Vert\Delta_{j}\rho(t)-\Delta_{j}\rho(t^{\prime})\Vert_{L^{2}}&\le& \int_{t^\prime}^{t}\Vert\Delta\Delta_{j}\rho(\tau)\Vert_{L^2}d\tau+\int_{t^\prime}^{t}\Vert\Delta_{j}(v\cdot\nabla\rho)(\tau)\Vert_{L^2}d\tau\\
\nonumber &\lesssim& 2^{2j}\int_{t^\prime}^{t}\Vert\Delta_{j}\rho(\tau)\Vert_{L^2}d\tau+2^{\frac{5}{2}j}\int_{t^\prime}^{t}\Vert\Delta_{j}(v\rho)(\tau)\Vert_{L^1}d\tau\\
&\lesssim& \vert t-t^{\prime}\vert 2^{2j}\Vert\rho^{0}\Vert_{L^{2}}+\vert t-t^{\prime}\vert 2^{\frac{5}{2}j}\Vert v \Vert_{L^{\infty}_{t}L^{2}}\Vert\rho\Vert_{L^{\infty}_{t}L^{2}},
\end{eqnarray}
where we have used in the last line Proposition \ref{prop a1}-(a)
Putting now \eqref{s19} into \eqref{s18} and using Proposition \ref{prop a1}, we find  
\begin{eqnarray*}
\Vert \rho(t)-\rho(t^{\prime})\Vert_{B^{\frac{1}{2}}_{2,1}}&\lesssim& \vert t-t^{\prime}\vert 2^{\frac{5}{2}N}\Vert\rho^{0}\Vert_{L^{2}}+\vert t-t^{\prime}\vert 2^{3N}\Vert v \Vert_{L^{\infty}_{t}L^{2}}\Vert\rho\Vert_{L^{\infty}_{t}L^{2}}+\frac{\varepsilon}{2}\\
&\lesssim& \vert t-t^{\prime}\vert \Vert\rho^{0}\Vert_{L^2}\big(2^{\frac{5}{2}N}+C_{0}(1+t)2^{3N}\big)+\frac{\varepsilon}{2}.
\end{eqnarray*}
We choose $\vert t-t^{\prime}\vert<\gamma$ such that,
$$\vert t-t^{\prime}\vert \Vert\rho^{0}\Vert_{L^2}\big(2^{\frac{5}{2}N}+C_{0}(1+t)2^{3N}\big)<\frac{\varepsilon}{2}.$$
Finally,
$$\Vert \rho(t)-\rho(t^{\prime})\Vert_{B^{\frac{1}{2}}_{2,1}}<\varepsilon.$$
We will now prove the continuity in time for the density in $L^{p}$ space. Denote $S(t):=e^{t\Delta}$ is the semigroup of convolution defined by $$S(t)\rho(x)=(K_{t}\ast\rho)(x),$$ where $K_{t}=\frac{1}{t^{\frac{3}{2}}}K(\frac{x}{t^{\frac{1}{2}}})$ and $\widehat{K}=e^{-\vert\xi\vert^2}.$ It is well-known that $K$ is nonnegative and $\Vert K \Vert_{L^{1}}=1.$ We set now $g:=-v\cdot\nabla\rho.$ From the equation of $\rho,$ we get
\begin{eqnarray*}
\rho(t,x)&=& (K_{t}\ast\rho^{0})(x)+\int_{0}^{t}(K_{t-\tau}\ast g)(\tau)d\tau\\
&=& S(t)\rho^{0}(x)+\int_{0}^{t}S(t-\tau)g(\tau)d\tau.
\end{eqnarray*}
Thus from Duhamel formula, we have for every $0\le t^{\prime} \le t \le T,$
\begin{eqnarray}\label{s20}
\nonumber \rho(t,x)-\rho(t^{\prime},x)&=& (S(t)-S(t^{\prime}))\rho^{0}(x)+\int_{t^{\prime}}^{t}S(t-\tau)g(\tau)d\tau\\
&+& \int_{0}^{t^{\prime}}(S(t-\tau)-S(t^{\prime}-\tau))g(\tau)d\tau.
\end{eqnarray}
The first term of the right-hand side converge to zero as $t$ goes to $t^\prime$ since the map $t \longmapsto S(t)\rho^{0}$ is contnuous from $[0,+\infty[$ into $L^p$ with $1\le p<\infty.$
For the second term, we use Young and H\"older inequalities with $0\le t^{\prime}\le t \le T,$
\begin{eqnarray*}
\Vert\int_{t^{\prime}}^{t} S(t-\tau)g(\tau)d\tau\Vert_{L^p}&\le& \int_{t^{\prime}}^{t}\Vert S(t-\tau)g(\tau)\Vert_{L^p}d\tau\\
&\le& \int_{t^{\prime}}^{t}\Vert g(\tau)\Vert_{L^p}d\tau\\
&\le& \int_{t^{\prime}}^{t}\Vert v\cdot\nabla\rho(\tau)\Vert_{L^p}d\tau\\
&\le& \int_{t^{\prime}}^{t}\Vert v(\tau) \Vert_{L^\infty}\Vert\nabla\rho(\tau)\Vert_{L^p}d\tau\\
&\le& \Vert v \Vert_{L_{T}^{\infty}L^\infty}\int_{t^{\prime}}^{t}\Vert\nabla\rho(\tau)\Vert_{L^p}d\tau.
\end{eqnarray*}
Now since,
\begin{eqnarray*}
\int_{t^{\prime}}^{t}\Vert\nabla\rho(\tau)\Vert_{L^p}d\tau&\le& \Big\vert \int_{0}^{t}\Vert\nabla\rho(\tau)\Vert_{L^{p}}d\tau-\int_{0}^{t^{\prime}}\Vert\nabla\rho(\tau)\Vert_{L^{p}} d\tau\Big\vert\\
&\le& \int_{0}^{t}\Vert\nabla\rho(\tau)\Vert_{L^{p}}d\tau+\int_{0}^{t^{\prime}}\Vert\nabla\rho(\tau)\Vert_{L^{p}}d\tau.
\end{eqnarray*}
Then to estimate $\displaystyle\int_{0}^{t}\Vert\nabla\rho(\tau)\Vert_{L^{p}}d\tau$ we use \eqref{0x} and Bernstein inequality, we obtain
\begin{eqnarray}\label{s21}
\nonumber \int_{0}^{t}\Vert\nabla\rho(\tau)\Vert_{L^p}d\tau &\le& \sum_{j\ge -1}\int_{0}^{t}\Vert\nabla\Delta_{j}\rho(\tau)\Vert_{L^p}d\tau\\
\nonumber &\lesssim& \sum_{j\ge -1}2^{j}\int_{0}^{t}\Vert\Delta_{j}\rho(\tau)\Vert_{L^p}d\tau\\
\nonumber &\lesssim& \sum_{j\ge -1}2^{-j}\Vert\Delta_{j}\rho^{0}\Vert_{L^p}\big(1+\int_{0}^{t}\Vert\nabla v(\tau)\Vert_{L^\infty}d\tau\big)\\
&\lesssim&\Vert\rho^{0}\Vert_{L^p}(1+\int_{0}^{t}\Vert\nabla v(\tau)\Vert_{L^\infty}d\tau).
\end{eqnarray}
Similarly for $\displaystyle\int_{0}^{t^{\prime}}\Vert\nabla\rho(\tau)\Vert_{L^{p}}d\tau$ we find
$\displaystyle\int_{0}^{t^\prime}\Vert\nabla\rho(\tau)\Vert_{L^{p}}d\tau\lesssim\Vert\rho^{0}\Vert_{L^p}(1+\int_{0}^{t^\prime}\Vert\nabla v(\tau)\Vert_{L^\infty}d\tau).$
Using now the embedding $B_{2,1}^{\frac{5}{2}}\hookrightarrow L^{\infty}$, we find,
\begin{eqnarray*}
\Vert\int_{t^{\prime}}^{t} S(t-\tau)g(\tau)d\tau\Vert_{L^p}&\le& \Vert v \Vert_{L_{T}^{\infty}L^\infty}\int_{t^{\prime}}^{t}\Vert\nabla\rho(\tau)\Vert_{L^p}d\tau\\
&\lesssim& \Vert v \Vert_{L_{T}^{\infty}B^{\frac{5}{2}}_{2,1}}\Vert\rho^{0}\Vert_{L^p}\Big(1+\int_{0}^{t}\Vert\nabla v(\tau)\Vert_{L^\infty}d\tau+\int_{0}^{t^\prime}\Vert\nabla v(\tau)\Vert_{L^\infty}d\tau\Big).
\end{eqnarray*}
Hence the continuity in time is a consequence of Propositions \ref{prop a6} and \ref{prop a7}.
For the last term of \eqref{s20}, we use the identity,
$$S(t-\tau)g(\tau)-S(t^{\prime}-\tau)g(\tau)=-\int_{t^{\prime}-\tau}^{t-\tau}S(t^{\prime\prime})\Delta g(\tau)dt^{\prime\prime}.$$
Thus we have,
\begin{eqnarray*}
\Vert S(t-\tau)g(\tau)-S(t^{\prime}-\tau)g(\tau)\Vert_{L^p}&\le& \int_{t^{\prime}-\tau}^{t-\tau}\Vert S(t^{\prime\prime})\Delta g(\tau)\Vert_{L^p}dt^{\prime\prime}\\
&\le& \int_{t^{\prime}-\tau}^{t-\tau}\Vert S(t^{\prime\prime})\Delta_{-1}\Delta g(\tau)\Vert_{L^p}dt^{\prime\prime}\\
&+& \sum_{j\ge 0}\int_{t^{\prime}-\tau}^{t-\tau}\Vert S(t^{\prime\prime})\Delta_{j}\Delta g(\tau)\Vert_{L^p}dt^{\prime\prime}\\
&\le& \int_{t^{\prime}-\tau}^{t-\tau}\Vert K_{t^{\prime\prime}}\Vert_{L^1}\Vert\Delta_{-1}\Delta g(\tau)\Vert_{L^p}dt^{\prime\prime}\\
&+& \sum_{j\ge 0}\int_{t^{\prime}-\tau}^{t-\tau}\Vert e^{t^{\prime\prime}\Delta}\Delta_{j}\Delta g(\tau)\Vert_{L^p}dt^{\prime\prime}\\
&\le& C (t-t^{\prime})\Vert g(\tau)\Vert_{L^p}+C\sum_{j\ge 0}2^{2j}\int_{t^{\prime}-\tau}^{t-\tau}\Vert e^{-t^{\prime\prime}\Delta}\Delta_{j} g(\tau)\Vert_{L^p}dt^{\prime\prime}\\
&\lesssim& (t-t^{\prime})\Vert v(\tau) \Vert_{L^\infty}\Vert\nabla\rho(\tau)\Vert_{L^p}\\
&+& \sum_{j\ge 0}2^{2j}\int_{t^{\prime}-\tau}^{t-\tau}e^{-ct^{\prime\prime}2^{2j}}\Vert\Delta_{j} g(\tau)\Vert_{L^p}dt^{\prime\prime}\\
&\lesssim& (t-t^{\prime})\Vert v(\tau) \Vert_{L^\infty}\Vert\nabla\rho(\tau)\Vert_{L^p}\\
&+& \sum_{j\ge 0}2^{2j}\Vert\Delta_{j} g(\tau) \Vert_{L^p}\int_{t^{\prime}-\tau}^{t-\tau}e^{-ct^{\prime\prime}2^{2j}}dt^{\prime\prime}\\
&\lesssim& (t-t^{\prime})\Vert v(\tau) \Vert_{L^\infty}\Vert\nabla\rho(\tau)\Vert_{L^p}\\
&+& \sum_{j\ge 0} (1-e^{-c(t-t^{\prime})2^{2j}})\Vert\Delta_{j} g(\tau)\Vert_{L^p}.
\end{eqnarray*}
Where we have used Bernstein inequality and the following inequality proved in \cite{che99}: there exists $c, C>0$ such that for every $t>0, j\in\NN$ and $g\in L^p$, $1\le p\le\infty,$ 
$$\Vert e^{-t\Delta}\Delta_{j}g\Vert_{L^{p}}\le C e^{-ct2^{2j}}\Vert\Delta_{j}g\Vert_{L^{p}}.$$
Hence, it follows that,
\begin{eqnarray*}
\int_{0}^{t^{\prime}}\Vert (S(t-\tau)-S(t^{\prime}-\tau))g(\tau)\Vert_{L^{p}}d\tau&\le& C(t-t^{\prime})\Vert v \Vert_{L_{T}^{\infty}L^{\infty}}\Vert\nabla\rho\Vert_{L_{T}^{1}L^{p}}\\
&+& C\sum_{j\ge 0}\big(1-e^{-c(t-t^{\prime}) 2^{2j}}\big)\Vert\Delta_{j} g \Vert_{L_{T}^{1}L^{p}}.
\end{eqnarray*}
Since $1-e^{-x}\le x^{\beta}$ with $x\ge 0$ and $0\le\beta\le 1$ then we have for $0\le j$, $$1-e^{-c(t-t^{\prime}) 2^{2j}}\le c(t-t^{\prime})^{\frac{1}{4}} 2^{\frac{j}{2}}.$$
This gives
\begin{eqnarray*}
\int_{0}^{t^{\prime}}\Vert (S(t-\tau)-S(t^{\prime}-\tau))g(\tau)\Vert_{L^{p}}d\tau\lesssim (t-t^{\prime})\Vert v \Vert_{L_{T}^{\infty}L^{\infty}}\Vert\nabla\rho\Vert_{L_{T}^{1}L^{p}}+(t-t^{\prime})^{\frac{1}{4}}\Vert g \Vert_{L_{T}^{1}B_{p,1}^{\frac{1}{2}}}.
\end{eqnarray*}
Using Bony's decomposition, we can easily prove that,
\begin{eqnarray*}
\Vert g \Vert_{B_{p,1}^{\frac{1}{2}}}&\lesssim& \Vert v\cdot\nabla\rho\Vert_{B_{p,1}^{\frac{1}{2}}}\\
&\lesssim& \Vert v \Vert_{B_{\infty,1}^{1}}\Vert\rho\Vert_{B_{p,1}^{\frac{3}{2}}}.
\end{eqnarray*}
Therefore,
$$\Vert g \Vert_{L_{T}^{1}B_{p,1}^{\frac{1}{2}}}\lesssim \Vert v \Vert_{L_{T}^{\infty}B_{\infty,1}^{1}}\Vert\rho\Vert_{L_{T}^{1}B_{p,1}^{\frac{3}{2}}}.$$
To estimate $\Vert\rho\Vert_{L_{T}^{1}B_{p,1}^{\frac{3}{2}}}$ we use as before \eqref{0x}, Proposition \ref{prop a1}-a) and Bernstein inequality and get, 
\begin{eqnarray}\label{s22}
\nonumber \Vert\rho\Vert_{L_{T}^{1}B_{p,1}^{\frac{3}{2}}}&=&C\Vert\Delta_{-1}\rho\Vert_{L_{T}^{1}L^{p}}+\sum_{j\ge 0}2^{\frac{3}{2}j}\Vert\Delta_{j}\rho\Vert_{L_{T}^{1}L^{p}}\\
\nonumber &\lesssim& \Vert\rho^{0}\Vert_{L^{p}}T+\sum_{j\ge 0}2^{-\frac{1}{2}j}\Vert\Delta_{j}\rho^{0}\Vert_{L^{p}}(1+\int_{0}^{T}\Vert\nabla v(\tau)\Vert_{L^{\infty}}d\tau)\\
&\lesssim& \Vert\rho^{0}\Vert_{L^{p}}(1+T+\int_{0}^{T}\Vert\nabla v(\tau)\Vert_{L^{\infty}}d\tau).
\end{eqnarray}
Using the embedding $B_{\infty,1}^{1}\hookrightarrow L^{\infty},$ \eqref{s21} and \eqref{s22}, we get
\begin{eqnarray*}
\int_{0}^{t^{\prime}}\Vert (S(t-\tau)-S(t^{\prime}-\tau))g(\tau)\Vert_{L^p}d\tau&\lesssim& (t-t^{\prime})\Vert v \Vert_{L_{T}^{\infty}L^{\infty}}\Vert\nabla\rho\Vert_{L_{T}^{1}L^{p}}\\
&+& (t-t^{\prime})^{\frac{1}{4}}\Vert v \Vert_{L_{T}^{\infty}B_{\infty,1}^{1}}\Vert\rho\Vert_{L_{T}^{1}B_{p,1}^{\frac{3}{2}}}\\
&\lesssim& \Big((t-t^{\prime})\Vert\nabla\rho\Vert_{L_{T}^{1}L^{p}}+(t-t^{\prime})^{\frac{1}{4}}\Vert\rho\Vert_{L_{T}^{1}B_{p,1}^{\frac{3}{2}}}\Big)\Vert v \Vert_{L_{T}^{\infty}B_{\infty,1}^{1}}\\
&\lesssim& \Big((t-t^{\prime})+(t-t^{\prime})^{\frac{1}{4}}\Big)\Vert\rho^{0}\Vert_{L^p}\Big(1+T\\
&+& \int_{0}^{T}\Vert\nabla v(\tau)\Vert_{L^{\infty}}d\tau\Big)\Vert v \Vert_{L_{T}^{\infty}B^{1}_{\infty,1}}.
\end{eqnarray*}
The norms in the right-hand side are finite according to \eqref{s5} and Proposition \ref{prop a6}.\\
This achieves the proof of the continuity in time of the density $\rho$ in $L^p$ space.\\
As we have proved the continuity in time of the density $\rho$ in $L^p$ space, we can prove in the same way the continuity in time of the second moment of the density in $L^2$ space. For this we will estimate only one term and the other terms will be exactly the same estimate. Recall now from \eqref{66}, the second moment of $\rho$ satisfies the following equation, with $G:=\vert x_{h}\vert^{2}\rho$ and $f:= x_{h}\rho,$
$$\partial_{t}G-\Delta G=-v\cdot\nabla G+2v^{h}f-2\nabla_{h}\rho-4\textnormal{div}_{h}f:=F.$$
Then we can easily see with $0\le t^{\prime}\le t\le T$ that,
\begin{eqnarray*}
G(t,x)-G(t^\prime,x)=(S(t)-S(t^\prime))G^{0}(x)+\int_{t^\prime}^{t}S(t-\tau)F(\tau)d\tau+\int_{0}^{t^\prime}(S(t-\tau)-S(t^{\prime}-\tau))F(\tau)d\tau.
\end{eqnarray*}
We will use to control the second and the third terms in $L^2$ space an integration by parts and by using H\"older inequality. We will be going to work with $\int_{t^\prime}^{t}\Vert S(t-\tau)v\cdot\nabla G(\tau)\Vert_{L^2}d\tau.$ First we write by an integration by parts:
\begin{eqnarray*}
S(t-\tau)v\cdot\nabla G &=& S(t-\tau)\textnormal{div}(v\cdot G)\\
&=& \sum_{i,j}K_{t-\tau}\ast\partial_{i}(v^{i}\,G^{j})\\
&=& \sum_{i,j}\partial_{i}K_{t-\tau}\ast(v^{i}\,G^{j}).
\end{eqnarray*}
Thus by convolution inequality, we have
\begin{eqnarray*}
\Vert S(t-\tau)v\cdot\nabla G \Vert_{L^2}&\le& \sum_{i,j}\Vert\partial_{i}K_{t-\tau}\ast(v^{i}\,G^{j})\Vert_{L^2}\\
&\le& \sum_{i,j}\Vert\partial_{i}K_{t-\tau}\Vert_{L^1}\Vert v^{i}\,G^{j}\Vert_{L^2}.
\end{eqnarray*}
Now since $K_{t-\tau}=\frac{1}{(t-\tau)^{\frac{3}{2}}}K(\frac{x}{(t-\tau)^{\frac{1}{2}}}),$ then
$$\vert\partial_{i}K_{t-\tau}\vert\le\frac{1}{(t-\tau)^{\frac{3}{2}}}\frac{1}{(t-\tau)^{\frac{1}{2}}}\Big\vert\partial_{i}K(\frac{x}{(t-\tau)^{\frac{1}{2}}})\Big\vert.$$
Thus
\begin{eqnarray*}
\Vert\partial_{i}K_{t-\tau}\Vert_{L^1}&\le&  \frac{1}{(t-\tau)^{2}}\Vert\partial_{i}K(\frac{x}{(t-\tau)^{\frac{1}{2}}})\Vert_{L^1}\\
&\le& \frac{1}{(t-\tau)^{2}}(t-\tau)^{\frac{3}{2}}\Vert\partial_{i}K\Vert_{L^1}\\
&\le& \frac{1}{(t-\tau)^{\frac{1}{2}}}\Vert\partial_{i}K\Vert_{L^1}.
\end{eqnarray*}
This gives that,
\begin{eqnarray*}
\Vert S(t-\tau)v\cdot\nabla G \Vert_{L^2}\lesssim\Vert v\,G \Vert_{L^2}\frac{1}{(t-\tau)^{\frac{1}{2}}}.
\end{eqnarray*}
Therefore by using H\"older's inequality, we obtain that,
\begin{eqnarray*}
\int_{t^\prime}^{t}\Vert S(t-\tau)v\cdot\nabla G(\tau)\Vert_{L^2}d\tau&\le& \int_{t^\prime}^{t}\frac{1}{(t-\tau)^{\frac{1}{2}}}\Vert v(\tau)\Vert_{L^\infty}\Vert G(\tau)\Vert_{L^2}d\tau\\
&\lesssim& (t-t^\prime)^{\frac{1}{2}}\Vert v \Vert_{L_{T}^{\infty}L^\infty}\Vert G \Vert_{L_{T}^{\infty}L^2}\\
&\lesssim& (t-t^\prime)^{\frac{1}{2}}\Vert v \Vert_{L_{T}^{\infty}B^{1}_{\infty,1}}\Vert G \Vert_{L_{T}^{\infty}L^2}.
\end{eqnarray*} 
The continuity in time follows from \eqref{s5} and Proposition \ref{p1}-b).

\end{document}